\documentclass{siamart220329}
\usepackage{amsmath}
\usepackage{amssymb}
\usepackage{bm}
\usepackage{xcolor}
\usepackage{mathabx}    

\usepackage{cases}
\usepackage{empheq}
\usepackage{lipsum}
\usepackage{amsfonts}
\usepackage{graphicx}
\usepackage{epstopdf}
\usepackage{algorithmic}
\usepackage{tikz}
\usepackage{booktabs}
\usetikzlibrary{shapes,arrows}
\usepackage{microtype}
\ifpdf
\DeclareGraphicsExtensions{.eps,.pdf,.png,.jpg}
\else
\DeclareGraphicsExtensions{.eps}
\fi

\usepackage{amsopn}

\usepackage{lipsum}
\usepackage{amsfonts}
\usepackage{graphicx}
\usepackage{epstopdf}
\usepackage{algorithmic}
\usepackage{stmaryrd}

\ifpdf
  \DeclareGraphicsExtensions{.eps,.pdf,.png,.jpg}
\else
  \DeclareGraphicsExtensions{.eps}
\fi

\newsiamremark{remark}{Remark}
\newsiamremark{hypothesis}{Hypothesis}
\crefname{hypothesis}{Hypothesis}{Hypotheses}
\newsiamthm{claim}{Claim}

\headers{Multilinear analysis of quaternion arrays}{J. Flamant, X. Luciani, S. Miron, Y. Zniyed}

\title{An Example Article\thanks{Submitted to the editors DATE.
\funding{This work was funded by the Fog Research Institute under contract no.~FRI-454.}}}

\usepackage{amsopn}
\DeclareMathOperator{\diag}{diag}

\foreach \x in {a,...,z}{
	\expandafter\xdef\csname bf\x \endcsname{\noexpand\ensuremath{\noexpand\mathbf{\x}}}
	\expandafter\xdef\csname bm\x \endcsname{\noexpand\ensuremath{\noexpand\boldsymbol{\x}}}
}
\foreach \x in {A,...,Z}{
	\expandafter\xdef\csname bs\x \endcsname{\noexpand\ensuremath{\noexpand\boldsymbol{\x}}}
	\expandafter\xdef\csname bf\x \endcsname{\noexpand\ensuremath{\noexpand\mathbf{\x}}}
	\expandafter\xdef\csname bb\x \endcsname{\noexpand\ensuremath{\noexpand\mathbb{\x}}}
	\expandafter\xdef\csname ds\x \endcsname{\noexpand\ensuremath{\noexpand\mathds{\x}}}
	\expandafter\xdef\csname cal\x \endcsname{\noexpand\ensuremath{\noexpand\mathcal{\x}}}
}

\def\bmmu{\boldsymbol{\mu}}

\newcommand{\tens}[1]{\bm{\mathcal{#1}}} 
\newcommand{\tenselem}[1]{\mathcal{#1}}  

\newcommand{\bPhi}{\bm{\Phi}}

\newcommand{\bPsi}{\bm{\Psi}}
\newcommand{\bOm}{\bm{\Omega}}

\newcommand{\bbHl}{\bbH_{\scriptscriptstyle\rm L}}
\newcommand{\bbHr}{\bbH_{\scriptscriptstyle\rm R}}

\newcommand{\T}{{\sf T}}        
\renewcommand{\H}{{\sf H}}      
\makeatletter
\newcommand{\col}[1]{\mathop{\operator@font col}\{#1\}}   

\DeclareMathOperator{\rank}{rank}
\DeclareMathOperator{\rrank}{rank_{\scriptscriptstyle \rm R}}
\DeclareMathOperator{\lrank}{rank_{\scriptscriptstyle \rm L}}
\DeclareMathOperator{\krank}{k}
\DeclareMathOperator{\rkrank}{k_{\scriptscriptstyle \rm R}}
\DeclareMathOperator{\lkrank}{k_{\scriptscriptstyle \rm L}}
\DeclareMathOperator{\kprank}{k'}
\DeclareMathOperator{\Span}{span}
\DeclareMathOperator{\rspan}{span_{\scriptscriptstyle \rm R}}
\DeclareMathOperator{\lspan}{span_{\scriptscriptstyle \rm L}}

\newcommand{\re}[1]{\mathop{\operator@font Re} \left\lbrace#1\right\rbrace}   
\newcommand{\vecop}[1]{\mathop{\operator@font vec} \{#1\}}  
\newcommand{\mini}[1]{\mathop{\operator@font min}\{#1\}}  
\newcommand{\maxi}[1]{\mathop{\operator@font max}\{#1\}}     

\newcommand{\trace}[1]{\mathop{\operator@font trace}\{#1\}}   
\newcommand{\Diag}[1]{\mathop{\operator@font Diag}\{#1\}}    
\newcommand{\argmin}{\mathop{\operator@font arg\,min}}

\newcommand{\offdiag}[1]{\mathop{\operator@font offdiag}\{#1\}}    
\newcommand{\Proj}[2]{\mathop{\operator@font Proj_{#1}}{#2}}
\newcommand{\ProjGrad}[2]{\mathop{{\operator@font Proj} \nabla}#1(#2)}

\def\real{\mathrm{Re}\,}
\def\imag{\mathrm{Im}\,}

\newcommand{\conj}[1]{\overline{#1}}

\newcommand{\bbCi}{\bbC_{\bmi}}

\newcommand{\lmul}{\cdot_{\scriptscriptstyle \rhd} }
\newcommand{\rmul}{\cdot_{\scriptscriptstyle\lhd}}

\newcommand{\nmod}{\times}
\newcommand{\lnmod}{\times^{\scriptscriptstyle\rm \rhd}}
\newcommand{\rnmod}{\times^{\scriptscriptstyle\rm \lhd}}

\newcommand{\krb}{\mathop{\odot}}   
\newcommand{\lkrb}{\mathop{\odot}_{\scriptscriptstyle \rhd}}
\newcommand{\rkrb}{\mathop{\odot}_{\scriptscriptstyle \lhd}}
      
\newcommand{\hadam}{\boxdot}    
\newcommand{\kron}{\mathop{\boxtimes}}   
\newcommand{\lkron}{\mathop{\boxtimes}_{\scriptscriptstyle \rhd}}   
\newcommand{\rkron}{\mathop{\boxtimes}_{\scriptscriptstyle \lhd}}   

\newcommand{\ladjC}{\chi_{\scriptscriptstyle\rhd}} 
\newcommand{\radjC}{\chi_{\scriptscriptstyle\lhd}} 
\newcommand{\ladjCp}{\ladjC^\pi}
\newcommand{\radjCp}{\radjC^\pi} 

\usepackage{hyperref}
\usepackage{cleveref}

\ifpdf
\hypersetup{
  pdftitle={Multilinear analysis of quaternion arrays},
  pdfauthor={J. Flamant, X. Luciani, S. Miron, Y. Zniyed}
}
\fi

\title{Multilinear analysis of quaternion arrays:\\theory and computation\thanks{Submitted to the editors on \today. \funding{this work was supported by CNRS Informatics through the TenQ project (2022-2023) and in part by the ANR grant RICOCHET
ANR-21-CE48-0013.} For the purpose of Open Access, a CC-BY public copyright licence has been applied by the authors to
the present document and will be applied to all subsequent versions up to the Author Accepted Manuscript arising from this
submission. All article content, except where otherwise noted, is licensed under a \href{https://creativecommons.org/licenses/by/4.0/}{Creative Commons Attribution (CC-BY) license.}}}
\author{Julien Flamant\thanks{Université de Lorraine, CNRS, CRAN, F-54000 Nancy, France} \and 
Xavier Luciani\thanks{Université de Toulon, Aix-Marseille Université, CNRS, LIS UMR 7020, France} \and Sebastian Miron\footnotemark[2] \and Yassine Zniyed\footnotemark[3]}

\begin{document}

\maketitle

\begin{abstract}
	Multidimensional quaternion arrays (often referred to as ``quaternion tensors") and their decompositions have recently gained increasing attention in various fields such as color and polarimetric imaging or video processing. Despite this growing interest, the theoretical development of quaternion tensors remains limited. This paper introduces a novel multilinear framework for quaternion arrays, which extends the classical tensor analysis to multidimensional quaternion data in a rigorous manner. Specifically, we propose a new definition of quaternion tensors as  $\bbH \bbR$-multilinear forms, addressing the challenges posed by the non-commutativity of quaternion multiplication. Within this framework, we establish the Tucker decomposition for quaternion tensors and develop a quaternion Canonical Polyadic Decomposition (Q-CPD). We thoroughly investigate the properties of the Q-CPD, including trivial ambiguities, complex equivalent models, and sufficient conditions for uniqueness. Additionally, we present two algorithms for computing the Q-CPD and demonstrate their effectiveness through numerical experiments. Our results provide a solid theoretical foundation for further research on quaternion tensor decompositions and offer new computational tools for practitioners working with quaternion multiway data.
\end{abstract}



\begin{table}[]
    \centering
    \begin{tabular}{p{0.3\linewidth} | p{0.6\linewidth}}
        \toprule
        Notation & Description\\
        \midrule
        $\bbR, \bbC, \bbH$& real, complex fields and quaternion skew-field\\
        $\bbC_{\bmmu} = \bbR \oplus \bmmu \bbR$ & complex subfield of $\bbH$, with $\bmmu^2=-1$\\
        $a, \bfa, \bfA, \tens{A}$ & scalar, vector, matrix and tensor\\
        $\conj{a}, \conj{\bfa}, \conj{\bfA}$ & conjugation of a scalar / vector / matrix\\
        $\bfA^\T, \bfA^\H$& transpose and conjugate-transpose\\
        $q = q_a + \bmi q_b + \bmj q_c + \bmk q_d $& $q\in \bbH$, $q_a, q_b, q_c, q_d$ real components\\
		$q = \vert q \vert \exp(\bmmu_q \varphi_q)$&polar form of $q\in \bbH$\\
        $\bfA \lmul \bfB, \bfA \rmul \bfB$ & direct and reverse quaternion matrix product\\
        $\ladjC(\bfA), \radjC(\bfA)$ & direct and reverse complex adjoint\\
		$\ladjCp(\bfA), \radjCp(\bfA)$ & direct and reverse column-wise complex adjoint\\
        $\bfA\kron \bfB, \bfA\lkron \bfB, \bfA\rkron \bfB$ & Kronecker product, direct and reverse Kronecker product\\
        $\bfA \krb \bfB, \bfA \lkrb \bfB, \bfA \rkrb \bfB$ & Khatri-Rao product, direct and reverse Khatri-Rao product\\
        $\bfA \hadam \bfB$ & Hadamard product\\
        $\bfT_{i::}, \bfT_{:j:}, \bfT_{::k}$ & $i$-th horizontal slice, $j$-th lateral slice, $k$-th frontal slice of tensor $\tens{T}$\\
        $\bfT_{(n)}$ &$n$-mode unfolding of tensor $\tens{T}$\\
        $\tens{T}\lnmod_n \bfA, \tens{T}\nmod_n \bfA, \tens{T}\rnmod_n \bfA$ & $n$-mode products with possible direct or reverse quaternion matrix multiplication\\
        \bottomrule
    \end{tabular}
    \caption{Notations used in this article. The matrices $\bfA$ and $\bfB$ have appropriate dimensions. When at least one if real-valued, the $\rhd$ and $\lhd$ indices are irrelevant and they are dropped in notations.}
    \label{tab:notations}
\end{table}

\section{Introduction}

Quaternions are famous in computer science by their ability to efficiently encode three-dimensional (3D) rotations (which is crucial for e.g., computer graphics \cite{shoemake1985animating}). 
In addition, recent years have seen an increased interest in quaternion arrays, where quaternions are used to encode as a single scalar a 3D or 4D real vector information. 
Important examples include RGB color images \cite{barthelemy_color_2015,chen_low-rank_2020,zou_quaternion_2016,zou_quaternion_2021}, polarimetric images \cite{flamant_quaternion_2020,pan2023separable}, vector array processing \cite{chen_augmented_2020,hobiger_multicomponent_2012} or joint wind and temperature forecasting \cite{took_quaternion_2009}, to cite only a few. 
For a recent review on the usefulness of quaternion representations in data science, signal and image processing, see \cite{miron2023quaternions} and the recent special issue \cite{ieeeSPM_special_issue_quaternions1,ieeeSPM_special_issue_quaternions2}.

Quaternion multiway arrays arise naturally when 3D or 4D vector data is collected across two or more diversities, such as time, space, wavelength, database index, etc. 
Over the last decades, an important research effort has been devoted to the study of quaternion vectors and matrices, see e.g., \cite{zhang_quaternions_1997} for an introduction to quaternion linear algebra. 
On the other hand, the study of quaternion higher-order multiway arrays (i.e., with 3 dimensions or more) remains in its infancy. 
In particular, the notion of \emph{quaternion tensor} -- as a multilinear mathematical object -- remains to be properly established, despite being commonly employed as a synonym for multidimensional arrays of quaternions. 
Several tensor-like tools for quaternion multiway arrays have been recently proposed \cite{qin2022singular, miao2023quaternion, pan2024block} and used in color imaging applications \cite{yang2024quaternion,miao2024quaternion}.
These works essentially build on generalization of the tensor singular
value decomposition (t-SVD) \cite{kilmer2011factorization,kilmer2013third} to the quaternion domain. 
In fact, while in the real and complex cases it is known \cite{gilman2022grassmannian} that such decompositions can be interpreted as a specific multilinear tensor decomposition (such as the Canonical Polyadic decomposition), the lack of a multilinear definition of quaternion tensors makes it an interesting open question.

This paper deals with the study of quaternion multidimensional arrays from a multilinear point of view to meaningfully define the notion of \emph{quaternion tensors}.
The main difficulty lies in handling the non-commutativity of quaternion multiplication, which complicates the analysis and prevents straightforward extensions of the usual definitions from the real and complex cases. 
It builds on our preliminary results in \cite{imhogiemhe2023low}, where only partial results were presented.
The contributions of this paper are threefold: \emph{(i)} we propose a coordinate-free definition of quaternion tensors as $\bbH\bbR$-multilinear forms and identify a quaternion multiway array with the representation of a quaternion tensor in some fixed product basis; \emph{(ii)} we introduce the notion of quaternion Tucker format, $n$-mode products, change-of-basis and characterize their properties; \emph{(iii)} we develop a complete theory for quaternion Canonical Polyadic Decomposition (Q-CPD), including characterization of trivial ambiguities, equivalent complex-valued models and establish several sufficient uniqueness conditions. 
Finally, two practical algorithms for computing the Q-CPD are presented and their performance is validated through numerical experiments.
We believe the results of this paper can be of interest to practitioners dealing with quaternion multiway arrays, by providing a general framework for the development of new quaternion tensor decompositions and relevant models for applications that require uniqueness guarantees. 
In addition, the theoretical study presented in this paper relies extensively on the necessary distinction between \emph{left} and \emph{right} linear independence of quaternion vectors -- due to quaternion non-commutativity -- leading us to introduce several new results on left and right ranks and Kruskal ranks of quaternion matrices, which may be of independent interest for researchers in quaternion linear algebra. 

This paper is organized as follows. 
\Cref{sec:preliminary} recalls basics about the quaternion algebra and introduces our notations. 
We then extend some key concepts of real and complex multilinear algebra to the quaternion case. 
In particular we establish several new results on the Kruskal rank of quaternion matrices that are crucial for our theoretical analysis of quaternion tensor decompositions. 
Our main contributions are exposed in subsequent sections.
\Cref{sec:quaternion_tensor_def} establishes a formal definition of a quaternion tensor as a multilinear form with a specific structure.
This enables a natural derivation of the Tucker model for quaternion tensors. 
\Cref{sec:quaternionCPD} focuses on the Q-CPD and the study of its properties, equivalent complex representations and provides several sufficient uniqueness conditions. 
\Cref{sec:algo} develops two different algorithms for computing the Q-CPD, which are further evaluated through numerical experiments in \Cref{sec:numerical_simulations}. 
\Cref{sec:concluson} concludes the paper. 
\Cref{app:technical} gathers technical details and proofs.

\section{Preliminaries}
\label{sec:preliminary}
This section starts by reviewing elementary properties of the set of quaternions $\bbH$. 
We proceed by introducing the basic notions of quaternion linear algebra, including quaternion matrix operations, complex adjoint(s), and quaternion vector spaces.
The theory in this paper requires us to distinguish between \emph{direct} and \emph{reverse} quaternion matrix products; together with the notion of \emph{left} and \emph{right} linear independence of quaternion vectors, these notions lead to several new results that may be of independent interest. These are presented in \Cref{subsub:rank_krank_quat_mat}.
In particular, \Cref{prop:left_right_ranks_complex_adj} links the right and left ranks of a quaternion matrix to those of its direct and reverse complex adjoints; \cref{def:krank_quaternion} defines the notion of left and right Kruskal rank ($k$-rank) for quaternion matrices, and \cref{lemma:kprank_krank_quat} establishes a crucial link between these $k$-ranks of a quaternion matrix and the so-called $k'$-rank of its direct and reverse complex adjoint.
For reference, \Cref{tab:notations} summarizes the notations used in this article.

\subsection{Quaternions}

Let $\bbR$ and $\bbC$ denote the usual real and complex field, respectively. 
The set of quaternions $\bbH$ defines a 4-dimensional normed division algebra over the real numbers $\bbR$.
It has canonical basis $\lbrace 1, \bmi, \bmj, \bmk\rbrace$, where $\bmi, \bmj, \bmk$ are imaginary units such that
\begin{equation}
    \bmi^2 = \bmj^2 = \bmk^2 = \bmi\bmj\bmk = -1, \quad \bmi\bmj = - \bmj \bmi, \bmi\bmj = \bmk\:.
\end{equation}
Any quaternion $q \in \bbH$ can be written as
\begin{equation}
    q = q_a + \bmi q_b + \bmj q_c + \bmk q_d\:,
\end{equation}
where $q_a, q_b, q_c, q_d \in \bbR$ are called the components of $q$.
The real part of $q$ is $\real q = q_a$ whereas its imaginary part is $\imag q = \bmi q_b+ \bmj q_c + \bmk q_d$.
A quaternion $q$ is said to be purely imaginary (or simply, pure) if $\real q = 0$.
The quaternion conjugate of $q$ is denoted by $\conj{q} = \real q - \imag q$. 
The modulus of $q$ is $\vert q \vert = \sqrt{q\conj{q}} = \sqrt{\conj{q}q} = \sqrt{q_a^2 + q_b^2 + q_c^2+q_d^2}$.
Any non-zero quaternion $q$ has an inverse $q^{-1} = \conj{q}/ \vert q \vert^2$.
Just like complex numbers, a quaternion can be written in polar form as $q = \vert q \vert \exp(\bmmu_q \varphi_q)$, where its axis $\bmmu_q$ is a pure unit quaternion (i.e., such that $\bmmu_q^2 = -1)$) and where $\varphi_q \in \bbR$ is the angle of $q$.

One of the specific features of quaternions is the non-commutativity of the quaternion product.
Indeed, for $p, q \in \bbH$, one has $pq \neq qp$ in general since imaginary units do not commute. 
Nonetheless, standard operations such as conjugation and inversion distribute well over products of two quaternions: one has $\conj{(pq)} = \conj{q}\:\conj{p}$ and $(pq)^{-1} = q^{-1}p^{-1}$.

Given a pure unit quaternion $\bmmu$ (i.e., such that $\real \bmmu = 0$ and $\vert \bmmu \vert = 1$), one can define the set $\bbC_{\bmmu} = \bbR \oplus \bmmu \bbR$ which is a \emph{complex subfield} of $\bbH$. 
In particular, any quaternion $q$ can be written as a combination of two complex numbers in $\bbC_{\bmmu}$, a procedure known as the \emph{Cayley-Dickson decomposition} of $q$. 
For $\bmmu = \bmi$, it reads 
\begin{equation}
	q = q_1 + q_2\bmj, \quad q_1, q_2 \in \bbC_{\bmi}.\label{eq:CD_decomposition}
\end{equation}
The decomposition \eqref{eq:CD_decomposition} permits to establish a useful complex matrix representation of quaternions arrays \cite{zhang_quaternions_1997}, known as the complex matrix adjoint.
This will be investigated in detail in \Cref{subsub:complex_adjoints}.

\subsection{Quaternion linear algebra}
This section reviews important notions of linear algebra in the quaternion domain. For a comprehensive treatment of the subject, see e.g., \cite{zhang_quaternions_1997} or \cite{rodman2014topics}. 
In addition, we prove some new results that will be useful to study the uniqueness of quaternion tensor decompositions in \Cref*{sec:quaternionCPD}.

\subsubsection{Quaternion matrix operations}
Quaternion vectors and matrices  can be defined as 1-D and 2-D arrays with quaternion entries. 
The $i$-th entry of quaternion vector $\bfq \in \bbH^N$ is given by $(\bfq)_n = q_n \in \bbH$, while for a quaternion matrix $\bfA \in \bbH^{M\times N}$, its $(m, n)$-th entry reads $(\bfA)_{mn} = a_{mn} \in \bbH$.
The columns of $\bfA \in \bbH^{M\times N}$ are denoted by $\bfa_1, \bfa_2, \ldots, \bfa_N \in \bbH^M$. 
The transpose of $\bfA \in \bbH^{M\times N}$ is denoted by $\bfA^\T \in \bbH^{N\times M}$ and its conjugate-transpose is $\bfA^\H := \conj{(\bfA^\T)} = (\conj{\bfA})^\T \in \bbH^{N\times M}$.
Finally, we denote by $\bfI_M$ the identity matrix in $\bbH^{M\times M}$, which is the same as in $\bbC^{M\times M}$ and $\bbR^{M\times M}$.

Standard operations of linear algebra translate to the quaternion domain provided that operations take into account the non-commutativity of the quaternion product. 
Consider $\bfA \in \bbH^{M\times N}$ and $\bfB \in \bbH^{N\times P}$. Then it is possible to define the direct and reverse quaternion matrix products as 
\begin{align}
	 \text{direct quaternion matrix product}\quad & (\bfA \lmul \bfB)_{mp} := \sum_{n=1}^N a_{mn}b_{np} ,\label{eq:left_matrix_mul} \\
	 \text{reverse quaternion matrix product}\quad & (\bfA \rmul \bfB)_{mp} := \sum_{n=1}^N b_{np}a_{mn} \label{eq:right_matrix_mul}.
\end{align}
The direct or reverse denomination simply refers to the ordering of the two factors in the expression of the matrix product entries: direct refers to the left-to-right direction, while reverse means right-to-left direction.
Note that, in many publications, the direct matrix product is often implicitly used \cite{zhang_quaternions_1997}; however this distinction will appear crucial to our construction of quaternion tensors (see \Cref{sec:quaternion_tensor_def}). 

Due to quaternion non-commutativity, for two matrices with compatible sizes, $\bfA \lmul \bfB \neq \bfA \rmul \bfB$ in general; however, if one of the matrices has real-valued entries (say e.g., $\bfB \in \bbR^{N \times P}$), then $\bfA \lmul \bfB  = \bfA \rmul \bfB$ and we can drop the subscript to write $\bfA \bfB$ if there is no risk of confusion from the context.

Following \cite{schulz_using_2014}\footnote{In \cite{schulz_using_2014}, the authors refer to direct and reverse quaternion matrix products as ``left'' and ``right'', respectively. To avoid confusion with left and right quaternion vector spaces introduced in the next section, we chose this alternate denomination of quaternion matrix products.}, the direct and reverse quaternion matrix products satisfy the following properties: 
\begin{align*}
	\text{ transposition }\qquad  & (\bfA \lmul \bfB)^\T = \bfB^\T \rmul \bfA^\T &\text{ and }& (\bfA \rmul \bfB)^\T = \bfB^\T \lmul \bfA^\T  ,\\
	\text{ conjugation }\qquad  & \overline{(\bfA \lmul \bfB)} = \conj{\bfA} \rmul \conj{\bfB} &\text{ and }& \conj{(\bfA \rmul \bfB)} = \conj{\bfA} \lmul \conj{\bfB}  ,\\
	\text{ conjugate-transposition }\qquad  & (\bfA \lmul \bfB)^\H = \bfB^\H \lmul \bfA^\H &\text{ and }&(\bfA \rmul \bfB)^\H = \bfB^\H \rmul \bfA^\H .
\end{align*}
Proofs of these properties are readily obtained from the direct and reverse quaternion matrix product definitions.

The same need to distinguish between direct and reverse product applies to Kronecker and Khatri-Rao products, which are two standard matrix products encountered in tensor algebra. 
Given two quaternion matrices $\bfA \in \bbH^{M\times N}$ and $\bfB \in \bbH^{P\times Q}$, their direct and reverse Kronecker product are defined as: 
\begin{align*}
	\bfA \lkron \bfB := \begin{bmatrix}
		a_{11}\bfB & \ldots & a_{1N}\bfB  \\
		\vdots & \ddots & \vdots\\
		a_{M1}\bfB &  \ldots & a_{MN}\bfB
	\end{bmatrix} \in \bbH^{(MP)\times (NQ)},  \\ 
	\bfA \rkron \bfB := \begin{bmatrix}
		\bfB a_{11} & \ldots & \bfB a_{1N}\\
		\vdots & \ddots & \vdots\\
		\bfB a_{M1}&  \ldots & \bfB a_{MN}
	\end{bmatrix}\in \bbH^{(MP)\times (NQ)} .
\end{align*}
Similarly, given two matrices $\bfA \in \bbH^{M\times P}$ and $\bfB \in \bbH^{N\times P}$, their direct and reverse Khatri-Rao product (column-wise Kronecker product) are defined as 
\begin{align*}
	\bfA \lkrb \bfB := \begin{bmatrix}
	\bfa_1 \lkron \bfb_1 & \bfa_2 \lkron \bfb_2 & \ldots \bfa_P \lkron \bfb_P
	\end{bmatrix} \in \bbH^{(MN)\times P} ,\\ 
	\bfA \rkrb \bfB := \begin{bmatrix}
		\bfa_1 \rkron \bfb_1 & \bfa_2 \rkron \bfb_2 & \ldots \bfa_P \rkron \bfb_P
		\end{bmatrix} \in \bbH^{(MN)\times P} . 
\end{align*}
Finally, the Hadamard product between two quaternion matrices $\bfA, \bfB \in \bbH^{M\times N}$ is defined as 
\begin{equation}
	(\bfA \hadam \bfB)_{mn} = a_{mn} b_{mn} \; .
\end{equation}
which is naturally non-commutative, unlike its real and complex counterparts. Therefore, we do not distinguish between direct and reverse product since it simply corresponds to exchanging the positions of the two matrices $\bfA$ and $\bfB$. 

Finally, note that similarly to the quaternion matrix product case, the distinction between direct and reverse operations becomes irrelevant for Kronecker and Khatri-Rao products as long as one of the matrices is real-valued. In this case, we usually drop the subscript if there is no risk of confusion within the context. 
More properties on direct and reverse matrix operations can be found in \cite{schulz_using_2014}.

\subsubsection{Complex adjoints for quaternion matrix products}
\label{subsub:complex_adjoints}
The notion of complex adjoint of a quaternion matrix is a key tool in quaternion linear algebra \cite{zhang_quaternions_1997}.
In particular, it permits to establish the properties of quaternion matrices (rank, inversion, eigen-decompositions, etc.) from their complex adjoint counterpart. 
The definition of the complex adjoint is implicitly related to the choice of quaternion matrix multiplication, i.e., direct \eqref{eq:left_matrix_mul} or reverse \eqref{eq:right_matrix_mul}.
Therefore in full generality one needs to distinguish between \emph{direct} and \emph{reverse} adjoints.
Let $\bfA \in \bbH^{M\times N}$ and consider its Cayley-Dickson decomposition $\bfA = \bfA_1 + \bfA_2\bmj$, where $\bfA_1, \bfA_2$ are matrices in $\bbC^{2M\times 2N}_{\bmi}$. 
The \emph{direct} complex adjoint $\ladjC(\bfA)$ and \emph{reverse} complex adjoint $\radjC(\bfA)$ enable an isomorphism between quaternion matrices (equipped with the corresponding matrix product) and complex matrices with block structure, such that
\begin{align}
	\ladjC(\bfA) = \begin{bmatrix}
       \bfA_1 & \bfA_2\\
        -\conj{\bfA}_2& \conj{\bfA}_1 
    \end{bmatrix}
    \in \bbC^{2M\times 2N}_{\bmi},\label{eq:left_complexAdjointMatrix}\\
	\radjC(\bfA) = \begin{bmatrix}
		\bfA_1& -\conj{\bfA}_2\\
		\bfA_2 & \conj{\bfA}_1 
	\end{bmatrix}
	\in \bbC^{2M\times 2N}_{\bmi}.\label{eq:right_complexAdjointMatrix}
\end{align}
The only difference -- yet crucial -- between \eqref{eq:left_complexAdjointMatrix} and \eqref{eq:right_complexAdjointMatrix} resides in the placement of anti-diagonal blocks.
\Cref{prop:complexAdjointProperties} below summarizes important properties of these complex adjoints.

\begin{proposition}[Properties of the complex adjoint]
Let $\bfA$ and $\bfB$ be two quaternion matrices with arbitrary dimensions.
The direct \eqref{eq:left_complexAdjointMatrix} and reverse \eqref{eq:right_complexAdjointMatrix} complex adjoint satisfy the following properties:
\begin{itemize}
	\item $\ladjC(\bfI_M) = \radjC(\bfI_M) = \bfI_{2M}$;
	\item $\ladjC(\bfA + \bfB) = \ladjC(\bfA) + \ladjC(\bfB)$ and $\radjC(\bfA + \bfB) = \radjC(\bfA) + \radjC(\bfB)$ if $\bfA, \bfB$ have the same size;
	\item $\ladjC(\bfA \lmul \bfB) = \ladjC(\bfA)\ladjC(\bfB)$ and $\radjC(\bfA \rmul \bfB) = \radjC(\bfA)\radjC(\bfB)$ if one can form the matrix product between $\bfA$ and $\bfB$;
	\item $\ladjC(\bfA^\H) = \ladjC^\H(\bfA)$ and $\radjC(\bfA^\H) = \radjC^\H(\bfA)$;
	\item $\ladjC(\bfA) = \radjC^\top(\bfA^\top)$;
\end{itemize}
	\label{prop:complexAdjointProperties}
\end{proposition}
The proof of these results follows by direct calculations. For additional properties, we refer the reader to \cite{zhang_quaternions_1997}, where the direct complex adjoint is considered - similar results hold for the reverse adjoint by simple adaption of the proofs therein.
\begin{remark}\label{remark:permuted_adjc}
	The standard definitions \eqref{eq:left_complexAdjointMatrix}--\eqref{eq:right_complexAdjointMatrix} of complex adjoints of a quaternion matrix can be related to the horizontal stack of the adjoint of each of its columns.
	Let $\bfA = [\bfa_1\: \bfa_2\:\ldots\:\bfa_N] \in \bbH^{M\times N}$, and introduce the  direct column-wise complex adjoint
	\begin{equation}
		\ladjCp(\bfA) = \begin{bmatrix}
			\ladjC(\bfa_1) \mid \ladjC(\bfa_2) \mid \ldots \mid\ladjC(\bfa_N)
		\end{bmatrix} \in \bbC^{2M\times 2N}_{\bmi}.\label{eq:definition_adjC_permuted}
	\end{equation}
	From the definition of the direct complex adjoint for a quaternion vector, it is easy to see that there exists a permutation matrix $\boldsymbol{\Pi} \in \bbR^{2N \times 2N}$ such that $\ladjC(\bfA) = \ladjCp(\bfA)\boldsymbol{\Pi}$. 
	Equation \eqref{eq:definition_adjC_permuted} also defines a uniform partition of  $\ladjCp(\bfA)$ in $N$ submatrices $\ladjC(\bfa_n) \in \bbCi^{2M\times 2}$.
	The same definition and properties holds for the columnwise reverse complex adjoint, which we denote by $\radjCp(\bfA)$. 
	This will be useful later on for the study of tensor decompositions of quaternion multiway arrays. 
\end{remark}

\subsubsection{Quaternion vector spaces}

The set of quaternions $\bbH$ is non-commutative and therefore it is not a field. 
As a result, since vector spaces are defined over a field (e.g., $\bbR$ or $\bbC$), it is not possible to define the notion of vector space over $\bbH$, \emph{per se}. 
Strictly speaking, one has to consider the notion of left- or right-module, which can be viewed as an extension of vector spaces where scalar multiplication is performed over a non-commutative ring.
However, in the case of quaternions, it is common practice \cite{rodman2014topics} to call left (resp. right) $\bbH$-modules left (resp. right) quaternion vector spaces, with little abuse. 
In particular, since $\bbH$ is a division algebra, quaternion vector spaces exhibit the usual properties of vector spaces (dimension, basis, linear maps, etc. ) up to the necessary distinction between left and right linearity.

For completeness, we give the definition of a left quaternion vector space below. The definition for a right vector space follows by direct adaption. 
\begin{definition}[Left quaternion vector space \cite{jamison1970extension}]
	A left vector space $\calH$ over $\bbH$ is an additive abelian group in which the operation of scalar multiplication by elements of $\bbH$ is defined. 
	Scalar multiplication is assumed to obey the following laws for all $\bfx, \bfy \in \calH$ and $q, p \in \bbH$: 
	\begin{itemize}
		\item $q(\bfx + \bfy) = q\bfx + q\bfy$;
		\item $(q+p)\bfx = q\bfx +p\bfx$;
		\item $(qp)\bfx = q(p\bfx)$;
		\item $1\cdot \bfx = \bfx$, where $1 \in \bbH$. 
	\end{itemize}
\end{definition}

A subscript ${}_{\scriptstyle  L}$ or ${}_{\scriptstyle  R}$ permits to distinguish between left and right quaternion vector spaces. 
Two important examples are: $\bbHl^M$, the left quaternion vector space of vectors in $\bbH^M$ and $\bbHr^M$, the right quaternion vector space of vectors in $\bbH^M$.

\subsubsection{The rank and Kruskal rank of quaternion matrices}
\label{subsub:rank_krank_quat_mat}

In this section we review the \emph{rank} of a quaternion matrix and extend the definition of the Kruskal rank, or simply $\krank$-rank, to quaternion matrices. 
In the real and complex cases, the rank of a matrix refers to the maximum number of \emph{linearly independent} columns (or rows) in the matrix.
When extending this notion to quaternion matrices, non-commutativity of quaternion multiplication  calls once again for a distinction between left and right linear independence. 
The rank of a quaternion matrix $\bfA$ is usually defined to be the maximum number of \emph{right} linearly independent columns \cite{zhang_quaternions_1997}.
However, as it will appear clearly in the sequel, the theory presented in this paper calls for a distinction between the \emph{left} and \emph{right} ranks (and thus, $\krank$-ranks) of a quaternion matrix. 
\begin{definition}[Left and right columns spans]
	Let $\bfA \in \bbH^{M\times N}$ and denote by $\bfa_1, \ldots, \bfa_N \in \bbH^M$ its columns. Then, the span of left-linear combination of columns of $\bfA$ is denoted as
	\begin{equation}
		\lspan(\bfA) := \left\lbrace \sum_{n=0}^N x_n \bfa_n \mid x_1, \ldots, x_n \in \bbH\right\rbrace,\label{eq:def:lspan}
	\end{equation}
	and the span of right-linear combinations of columns of $\bfA$ is given by
	\begin{equation}
		\rspan(\bfA) := \left\lbrace \sum_{n=0}^N \bfa_n x_n \mid x_1, \ldots, x_n \in \bbH\right\rbrace.\label{eq:def:rspan}
	\end{equation}
\end{definition}
From these definitions, one observes that $\lspan(\bfA)$ defines a left vector space over the columns of $\bfA$, while $\rspan(\bfA)$ defines a right vector space over the same columns. 

The following definition of left and right ranks of a quaternion matrix follows naturally. 

\begin{definition}[Left and right ranks of a quaternion matrix]
	Let $\bfA \in \bbH^{M\times N}$. The left rank of $\bfA$ is defined as $\lrank \bfA = \dim \lspan(\bfA)$ while the right rank is given by $\rrank \bfA = \dim \rspan(\bfA)$.
\end{definition}
\begin{remark}
	For a real or complex matrix $\bfA$, we drop the subscripts and keep the unambiguous notations $\Span(\bfA)$ to denote the span of its columns and its rank by $\rank\bfA$.
\end{remark}
As in the usual real and complex cases, the dimension of the left (resp. right) column span encodes the maximal number of left (resp. right) linearly independent vectors. 
For quaternion matrices, the left and right ranks are in general different; see e.g., \cite[Example 7.3]{zhang_quaternions_1997}.
Hence the distinction is critical. 
Nonetheless, the left and right ranks satisfy several important properties similar to their real and complex counterparts. 
In particular, for a quaternion matrix $\bfA\in \bbH^{M\times N}$, the rank cannot exceed its dimensions, i.e., $\lrank\bfA \leq \min (M, N)$ and $\rrank\bfA \leq \min (M, N)$. 
The matrix $\bfA$ is said to be full left rank if $\lrank \bfA = \min (M, N)$ and full right rank if $\rrank \bfA = \min (M, N)$. 

The following proposition relates the left and right ranks of a quaternion matrix to those of its direct and reverse complex adjoints. 
\begin{proposition}[Left and right ranks from complex adjoints]
	\label{prop:left_right_ranks_complex_adj}
	Let $\bfA \in \bbH^{M\times N}$ be a quaternion matrix.
	Then the left and right ranks of $\bfA$ can be obtained from direct and reverse complex adjoints as 
	\begin{enumerate}
	\item $\rank{\ladjC\left(\bfA\right)} = 2 \rrank \bfA$;
	\item $\rank{\radjC\left(\bfA\right)} = 2\lrank \bfA$.
	\end{enumerate}
\end{proposition}

\begin{proof}
	We only prove the first statement; the second follows from direct adaptation of the proof.
	Let $\bfA \in \bbH^{M\times N}$ and suppose that $\rrank \bfA = r \leq \min(M, N)$.\newline
	%
	Let us first prove that $\rank \ladjC(\bfA)\geq 2r$. 
	Since $\rrank \bfA = r$, this implies that there exist $r$ columns $\bfa_1, \ldots, \bfa_r$ of $\bfA$ which are right linearly independent. 
	Denote by $\bfA_r = [\bfa_1 \cdots \bfa_r] \in \bbH^{M\times r}$ this submatrix of $\bfA$.
	Now let $\boldsymbol{\beta} \in \bbCi^{2r}$ such that $\ladjC(\bfA_r)\boldsymbol{\beta} = 0$. 
	Partition $\boldsymbol{\beta} = [\boldsymbol{\beta}_1^\top\: -\overline{\boldsymbol{\beta}_2}^\top]^\top$ with $\boldsymbol{\beta}_1, \boldsymbol{\beta}_2 \in \bbCi^r$.
	The complex adjoint structure of $\ladjC(\bfA_r)$ implies that $\ladjC(\bfA_r)\boldsymbol{\beta}^\perp = 0$ as well, where $\boldsymbol{\beta}^\perp := [\boldsymbol{\beta}^\top_2\: \overline{\boldsymbol{\beta}_1}^\top]^\top$.
	Therefore, posing $\boldsymbol{\alpha} = \boldsymbol{\beta}_1 + \boldsymbol{\beta}_2 \bmj$ and observing that $[\boldsymbol{\beta}\: \boldsymbol{\beta}^\perp] = \ladjC(\boldsymbol{\alpha})$, one has $\ladjC(\bfA_r)\ladjC(\boldsymbol{\alpha}) = 0.$ 
	Using the properties of the direct adjoint
	we thus have $\ladjC(\bfA_r\lmul\boldsymbol{\alpha})=0$ and then $\bfA_r\lmul\boldsymbol{\alpha}=0.$
	Since we have supposed that the columns of $\bfA_r$ were linearly independent, we necessarily have $\boldsymbol{\alpha}=0$ hence $\boldsymbol{\beta}=0$. Therefore the $2r$ columns of $\ladjC(\bfA_r)$ are linearly independent over $\bbCi$.  Since $\ladjC(\bfA_r)$ is a submatrix of  $\ladjC(\bfA)$ we have $\rank \ladjC(\bfA) \geq 2r$.\newline
	%
	We now prove that $\rank \ladjC(\bfA) \leq 2r$. 
	Since $\rrank\bfA = r$, upon proper permutations of columns, one can factorize $\bfA = \bfA_r \lmul [\bfI_r \: \bfB]$, where $\bfA_r = [\bfa_1 \cdots \bfa_r] \in \bbH^{M\times r}$ has linearly independent columns and where $\bfB \in \bbH^{r\times (N-r)}$ encodes the right-linear combinations that yields the remaining $N-r$ columns of $\bfA$. 
	Applying the direct complex adjoint to this factorization, one gets
	$\ladjC(\bfA) = \ladjC(\bfA_r)\ladjC([\bfI_r \: \bfB])$. 
	This shows that $\ladjC(\bfA)$ can always be factorized as a product of a $2M\times 2r$ matrix by a $2r\times 2N$ matrix, hence $\rank \ladjC(\bfA) \leq 2r$. 
	Bringing the two parts together yields $\rank \ladjC(\bfA) = 2r = 2\rrank \bfA$.
\end{proof}

The definition of the left and and right Kruskal ranks follows naturally from that of the left and right ranks of a quaternion matrix. 

\begin{definition}[Left and right Kruskal ranks of a quaternion matrix]\label{def:krank_quaternion}
    Let $\bfA \in \bbH^{M\times N}$ be a quaternion matrix. Then:
	\begin{itemize}
		\item the left Kruskal rank of $\bfA$ is denoted by $\lkrank(\bfA)$. It is the maximal number $r$ such that any set of $r$ columns of $\bfA$ are \emph{left} linearly independent. 
		\item the right Kruskal rank of $\bfA$ is denoted by $\rkrank(\bfA)$. It is the maximal number $r$ such that any set of $r$ columns of $\bfA$ are \emph{right} linearly independent. 
	\end{itemize}
\end{definition}
For a real or complex matrix $\bfA$, we simply write $\krank(\bfA)$ to denote its usual Kruskal rank. For quaternion matrices, we use the shorthand  $\lkrank$-rank and $\rkrank$-rank to mean the left and right Kruskal ranks, respectively.

It is possible to relate the $\lkrank$-rank and $\rkrank$-rank of a quaternion matrix to the $\kprank$-rank of its direct and reverse column-wise complex adjoints (see definitions in \Cref{remark:permuted_adjc}).
Let us first recall the definition of the $\kprank$-rank of a partitioned matrix. 

\begin{definition}[$\kprank$-rank of a partitioned matrix \cite{de_lathauwer_decompositions_2008PartI}]
    Let $\bfX \in \bbF^{M\times N}$ be a real or complex partitioned matrix (not necessarily uniformly). 
    Then the $\kprank$-rank of $\bfX$, denoted by $\kprank(\bfX)$, is the maximal number $r$ such that any set of $r$ submatrices of $\bfX$ yields a set of linearly independent columns.
    \label{def:kp}
\end{definition}

The following lemma establishes the relation between the $\lkrank$-rank and $\rkrank$-rank of a quaternion matrix and the $\kprank$-rank of its column-wise complex adjoints. 
\begin{lemma}[$\kprank$-rank of column-wise adjoint matrices]\label{lemma:kprank_krank_quat}
Let $\bfA \in \bbH^{M\times N}$ and let $\ladjCp(\bfA) \in \bbCi^{2M\times 2N}$ and $\radjCp(\bfA) \in \bbCi^{2M\times 2N}$ be its column-wise direct complex adjoint and reverse complex adjoint, respectively, with the uniform partition introduced in \Cref{remark:permuted_adjc}.
Then $\rkrank(\bfA) = \kprank (\ladjCp(\bfA))$ and $\lkrank(\bfA) = \kprank(\radjCp(\bfA))$.
\end{lemma}
\begin{proof}
	We start by proving the result for the $\rkrank$-rank of $\bfA$.
	Let us assume that $\rkrank(\bfA)=c$. 
	According to the definition of the $\rkrank$-rank, any subset of $c$ columns of $\bfA$ forms a submatrix $\bfA_c$ of $\bfA$ such that $\rrank \bfA_c = c$.
	Using \Cref{prop:left_right_ranks_complex_adj}, we obtain $ \rank {\ladjC(\bfA_c)} =2c$. 
	This is equivalent to saying that any set of direct adjoint matrices associated with the $c$ columns of $\bfA_c$ yields a set of linearly independent columns of $\ladjCp(\bfA)$, which, by Definition \ref{def:kp} implies $\kprank(\ladjCp(\bfA)) \geq c=\rkrank(\bfA)$.
	
	Using once again the definition of the $\rkrank$-rank, there exists at least a subset of  $c+1$ columns of  $\bfA$, such that the  matrix ${\bfA_{c+1}}$ formed by these $c+1$ columns is right rank deficient, { i.e.,} $\rrank {\bfA_{c+1}} < c+1$, implying  $\rank {\ladjC(\bfA_{c+1})} < 2(c+1)$. Consequently, as $\ladjC(\bfA) = \ladjCp(\bfA)\boldsymbol{\Pi}$, there exists a set of $c+1$ submatrices of $\ladjCp(\bfA) $ that yields a set of linearly dependent columns, thereby completing the proof. 
    The second result can be proved in an analogous way, by replacing $\ladjC(\bfA)$ by $\radjC(\bfA)$ and right ranks by left ranks.
	\end{proof}

This result will be essential to the study of uniqueness conditions for the quaternion Canonical Polyadic Decomposition, see further in \Cref{sec:quaternionCPD}. 
\section{Quaternion tensors: construction and properties}
\label{sec:quaternion_tensor_def}
In the standard real and complex cases, three alternative definitions of \emph{tensors} are possible \cite{lim2021tensors}: tensors as objects obeying certain transformation rules, tensors as multilinear maps or tensors as elements of tensor product spaces. 
While the first definition has historical importance, it has now been superseded by the two other definitions \cite{comon2014tensors}.
Over the past few years, there has been an increasing interest in studying multidimensional arrays of quaternions \cite{qin2022singular, miao2023quaternion, pan2024block, yang2024quaternion,miao2024quaternion}; while the term \emph{quaternion tensor} is often used, a proper algebraic, coordinate-free definition remains to be established.

In this paper, we choose to construct quaternion tensors as quaternion-valued multilinear forms. 
Our definition of a quaternion tensor can be viewed as an instance of \cite[Definition 3.3]{lim2021tensors} specified to quaternion-valued forms satisfying sufficient multilinearity properties.
More precisely, we identify a tensor of order $D$ with a quaternion-valued multilinear form $f: \calS_1 \times \calS_2 \times \cdots \times \calS_D \rightarrow \bbH$ where $\lbrace \calS_d \rbrace_{d=1}^D$ are some carefully chosen vector spaces. 
This definition enables a natural generalization of the notion of tensors to the quaternion domain, while revealing the crucial caveats associated to the non-commutativity of quaternion multiplication.
In what follows, \Cref{sub:quaternion_multilinearity} discusses the issue of quaternion multilinearity, \Cref{sub:HRmultilinearFunctionals} establishes the definition of quaternion tensors as multilinear forms.
\Cref{sub:transformationRules} derives elementary transformations for quaternion-valued tensors, which are direct consequences of \Cref{def:quaternionTensor} below.
Finally, \Cref{sub:tucker_format} introduces the Tucker format of a quaternion tensor.

\subsection{The quaternion multilinearity issue}
\label{sub:quaternion_multilinearity}
The definition of quaternion multilinearity is cumbersome due to the non-commutativity of quaternion multiplication. 
To illustrate this, considering the three-dimensional case is sufficient. 
Let $\calS_1,  \calS_2, \calS_3$ be three left quaternion vector spaces and let $f: \calS_1 \times \calS_2 \times \calS_3 \rightarrow \bbH$ be a multilinear form. 
By definition, $f$ is \emph{left}-linear in each one of its arguments, i.e., for any $\alpha, \beta \in \bbH$ and for any $\bfx_d, \bfy_d \in \calS_d$ with $d = 1, 2, 3$, one has
\begin{align}
     f(\alpha \bfx_1 +\beta \bfy_1, \bfx_2, \bfx_3) = \alpha f(\bfx_1, \bfx_2, \bfx_3) + \beta f(\bfy_1, \bfx_2, \bfx_3)   \label{eq:linearity1},   \\
     f(\bfx_1, \alpha \bfx_2 +\beta \bfy_2, \bfx_3) = \alpha f(\bfx_1, \bfx_2, \bfx_3) + \beta f(\bfx_1, \bfy_2, \bfx_3)     \label{eq:linearity2}, \\
     f(\bfx_1, \bfx_2, \alpha \bfx_3 +\beta \bfy_3) = \alpha f(\bfx_1, \bfx_2, \bfx_3) + \beta f(\bfx_1, \bfx_2, \bfy_3) \label{eq:linearity3}.
\end{align}
However, when one has to evaluate quantities such as $f(\alpha \bfx_1, \beta \bfx_2, \bfx_3)$, one is faced with a disturbing issue: in which order should linearity properties \eqref{eq:linearity1} -- \eqref{eq:linearity3} be applied?
Starting with  \eqref{eq:linearity1} and then \eqref{eq:linearity2} gives $f(\alpha \bfx_1, \beta \bfx_2, \bfx_3) = \alpha\beta f(\bfx_1, \bfx_2, \bfx_3)$, while the reversed order yields $f(\alpha \bfx_1, \beta \bfx_2, \bfx_3) = \beta\alpha f(\bfx_1, \bfx_2, \bfx_3)$.
Since the product of two quaternions is non-commutative, $\alpha\beta \neq \beta\alpha$ in general and thus $\beta\alpha f(\bfx_1, \bfx_2, \bfx_3) \neq  \alpha \beta f(\bfx_1, \bfx_2, \bfx_3)$.
Any other choice of variables leads to similar contradictions. 
In other terms, without specifying a somewhat arbitrary ordering of the variables $\bfx_1, \bfx_2$ and $\bfx_3$ (or equivalently, of quaternion vector spaces $\calS_1, \calS_2$ and $\calS_3$), the definition of quaternion multilinearity is unpractical as it is. 

The issue is not limited to left quaternion vector spaces. 
Considering $\calS_1, \calS_2, \calS_3$ to be right quaternion vector spaces leads to the same conclusions.
Mixing left and right quaternion vector spaces does not solve the issue either.
To see this, suppose without loss of generality  that $\calS_1, \calS_2$ are left quaternion vector spaces, while $\calS_3$ is a right quaternion vector space. 
Quaternion multilinearity between the first and last variables is unambiguous: indeed, for any $\alpha, \beta \in \bbH$, one has $f(\alpha \bfx_1, \bfx_2, \bfx_3 \beta) = \alpha f(\bfx_1, \bfx_2, \bfx_3 \beta) = f(\alpha \bfx_1, \bfx_2, \bfx_3 )\beta = \alpha f(\bfx_1, \bfx_2, \bfx_3 )\beta$. 
However, the first two modes suffer from the previous issue since for any $\alpha, \beta \in \bbH$, $f(\alpha \bfx_1, \beta \bfx_2, \bfx_3)$ can be equal to $\alpha\beta f(\bfx_1, \bfx_2, \bfx_3)$ or $\beta \alpha f(\bfx_1, \bfx_2, \bfx_3)$, depending on the order in which linearity properties are applied.

In the next section, we circumvent the issue of general quaternion multilinearity by introducing a special class of multilinear quaternion-valued forms. 
By carefully specifying the nature of the vector spaces $\lbrace \calS_d \rbrace$, we can recover fundamental multilinearity properties, similar to those in the real and complex cases. This class forms the foundation for the definition of quaternion tensors in Section \ref{sub:HRmultilinearFunctionals} and the establishment of their properties.

\subsection{Quaternion tensors as $\bbH\bbR$-multilinear forms}
\label{sub:HRmultilinearFunctionals}

In order to extend relevant multilinearity-like properties to the quaternion linear form $f: \calS_1 \times \cdots \times \calS_D \rightarrow \bbH$, it is necessary to restrict the nature of the different vector spaces  $\lbrace \calS_d\rbrace$ defining the domain of $f$.

\begin{definition}[$\bbH\bbR$-multilinear form]\label{def:HRmultilinerity}
	Let $D \geq 3$ be an integer and let the vector spaces $\lbrace \calS_d\rbrace$ be such that: $\calS_1$ is a \emph{left} quaternion vector space; $\calS_d$, $1 < d < D$ are {\em real} vector spaces; and such that $\calS_D$ is a \emph{right} quaternion vector space. 
	Then $f:\calS_1 \times \cdots\times \calS_D \rightarrow \bbH$ is said to be $\bbH\bbR$-multilinear if
	\begin{itemize}
		\item $f$ is left-quaternion linear in its first argument, i.e.,
			  \begin{equation}
			 \forall \alpha, \beta \in \bbH, \forall \bfx_1, \bfy_1 \in \calS_1, f(\alpha \bfx_1 + \beta \bfy_1, \ldots) = \alpha f( \bfx_1, \ldots) + \beta f( \bfy_1, \ldots) ;\label{eq:HRealMultiLin1}
			  \end{equation}
			\item $f$ is real-linear in its $d$-th argument, $1 < d < D$, i.e.,
			  \begin{equation}
				  \begin{split}
				 \forall \alpha, \beta \in \bbR, \forall \bfx_d, \bfy_d \in \calS_d, & \\
				 f(\ldots, \alpha\bfx_d + \beta\bfy_d, \ldots) &= f(\ldots, \bfx_d\alpha +  \bfy_d\beta, \ldots) \\
		&= \alpha f(\ldots, \bfx_d, \ldots)  + \beta f(\ldots, \bfy_d, \ldots)\\
&=  f(\ldots, \bfx_d, \ldots)\alpha  + f(\ldots, \bfy_d, \ldots)\beta;
				  \end{split}
			  \end{equation}
		\item $f$ is right-quaternion linear in its last argument, i.e., :
			  \begin{equation}
				  \forall \alpha, \beta \in \bbH, \forall \bfx_D, \bfy_D \in \calS_D, f(\ldots, \bfx_D\alpha + \bfy_D \beta) = f(\ldots, \bfx_D)\alpha + f(\ldots, \bfy_D)\beta.\label{eq:HRealMultiLin3}
			  \end{equation}
	\end{itemize}
\end{definition}

\begin{remark}
	Several remarks are in order. 
	The term $\bbH\bbR$-multilinearity refers to the fact that conditions \eqref{eq:HRealMultiLin1}--\eqref{eq:HRealMultiLin3} mix between real and quaternion linearity properties.
	In addition, note that if a function is $\bbH\bbR$-multilinear, by definition it is also $\bbR$-multilinear.
	Finally, it is worth noting that the ordering of vector spaces in Definition \ref{def:HRmultilinerity} is completely arbitrary, and can be permuted if necessary. The only restriction lies in considering exactly one quaternion left vector space, one quaternion right vector space, and $D-2$ real-vector spaces for an arbitrary number of dimensions $D \geq 3$. 
\end{remark}

In the sequel, we restrict our attention to finite dimensional vector spaces $\lbrace \calS_d\rbrace$ with respective dimension $\dim \calS_d = N_d$. 
Therefore, without loss of generality, we can assume that $\calS_1 = \bbHl^{N_1}$, $\calS_d = \bbR^{N_d}$ for $1 < d < D$ and $\calS_D = \bbHr^{N_D}$. 
The next definition defines quaternion tensors as $\bbH\bbR$-multilinear forms. 

\begin{definition}[Quaternion tensor]
	\label{def:quaternionTensor}
	Let $f: \bbHl^{N_1} \times \bbR^{N_2} \times \cdots \times \bbR^{N_{D-1}}\times \bbHr^{N_D}\rightarrow \bbH$ be a $\bbH\bbR$-multilinear form. 
	Let $\lbrace \bfe^{(d)}_{i_d}\rbrace_{i_d=1:N_d}$ denote a basis of the vector space $\calS_d$ for $1 \leq d \leq D$. 
	Then the quaternion tensor of order $D$ associated to $f$ is the $D$-dimensional quaternion array $\tens{T} \in \bbH^{N_1 \times \ldots \times N_D}$ defined by its entries as 
	\begin{equation}
		\tenselem{T}_{i_1i_2\ldots i_D} = f\left(\bfe^{(1)}_{i_1}, \bfe^{(2)}_{i_2}, \ldots, \bfe^{(D)}_{i_D}\right)\label{eq:quaternionTensorDefcoeff}
	\end{equation}
	i.e., $\tens{T}$ is the representation of $f$ in the Cartesian product basis $\bigtimes_{d=1}^D\lbrace \bfe^{(d)}_{i_d}\rbrace_{i_d=1:N_d}$.
\end{definition}

\begin{remark}
	Following standard practice in multilinear algebra, we identify a tensor with its multidimensional array coordinate representation in a given basis. 
\end{remark}

\Cref{def:quaternionTensor} suggests that different choices of bases for $\lbrace \calS_d\rbrace$ lead to different multidimensional quaternion arrays representations of the same $\bbH\bbR$-multilinear form $f$. 
Therefore a natural question is how these representations are related to one another. 
Let $f: \bbHl^{N_1} \times \bbR^{N_2} \times \cdots \times \bbR^{N_{D-1}}\times \bbHr^{N_D}\rightarrow \bbH$ be a $\bbH\bbR$-multilinear form and consider two quaternion multidimensional array representations:
\begin{align}
	\text{``new'' representation $\tens{T} \in  \bbH^{N_1 \times \ldots \times N_D}$ in the basis $\bigtimes_{d=1}^D\lbrace \bfe^{(d)}_{i_d}\rbrace_{i_d=1:N_d}$},\\
	\text{``old'' representation $\tens{\tilde{T}} \in  \bbH^{N_1 \times \ldots \times N_D}$ in the basis $\bigtimes_{d=1}^D\lbrace \tilde{\bfe}^{(d)}_{j_d}\rbrace_{j_d=1:N_d}$}.
\end{align}
Expressing ``new'' basis vectors as a function of ``old'' ones one gets
\begin{align}
    \bfe_{i_1}^{(1)}  &= \sum_{j_1=1}^{N_1} a^{(1)}_{i_1j_1} \tilde{\bfe}_{j_1}^{(1)}, & a^{(1)}_{i_1j_1}\in \bbH,\quad &i_1 = 1,  \ldots,  N_1   \label{eq:changeBasisFirstMode} \\
    \bfe_{i_d}^{(d)}  &= \sum_{j_d=1}^{N_d} a^{(d)}_{i_dj_d} \tilde{\bfe}_{j_d}^{(d)}, & a^{(d)}_{i_dj_d} \in \bbR, \quad & i_d = 1, \ldots, N_d \text{ and } 1 < d < D   \\
    \bfe_{i_D}^{(D)}  &= \sum_{j_D=1}^{N_D} \tilde{\bfe}_{j_D}^{(d)}a^{(D)}_{i_Dj_D},  & a^{(D)}_{i_Dj_D}\in \bbH, \quad & i_D = 1, \ldots, N_D, \label{eq:changeBasisLasttMode}
\end{align}
on account of the nature of the vector spaces $\lbrace \calS_d\rbrace$. 
Plugging expressions \eqref{eq:changeBasisFirstMode} -- \eqref{eq:changeBasisLasttMode} into the definition \eqref{eq:quaternionTensorDefcoeff} of $\tens{T}$, one obtains the fundamental change-of-basis relation for quaternion tensors relating the entries of the tensor $\tens{T}$ to that of $\tens{\tilde{T}}$:
\begin{equation}
	\tenselem{T}_{i_1i_2\ldots i_D}  = \sum_{j_1, j_2, \ldots,  j_D=1}^{N_1, N_2, \ldots, N_D} a_{i_1j_1}^{(1)} a_{i_1j_1}^{(2)} \ldots a_{i_{D-1}j_{D-1}}^{(D-1)}\tilde{\tenselem{T}}_{j_1j_2\ldots j_D} a_{i_Dj_D}^{(D)}.\label{eq:basischangeExplicit}
\end{equation}
On the one hand, remark the particular position of terms $a_{i_1j_1}^{(1)}$ and $a_{i_Dj_D}^{(D)}$, to the left and to the right of $\tilde{\tenselem{T}}_{j_1j_2\ldots j_D}$, respectively, due to the left (resp. right) linear nature of the associated quaternion vector spaces. 
On the other hand, the position of central terms $a_{i_dj_d}^{(d)}$, $1 < d < D$, with respect to $\tilde{\tenselem{T}}_{j_1j_2\ldots j_D}$ has no importance since they are real-valued. 

\subsection{Transformations rules for quaternion tensors}
\label{sub:transformationRules}
The change-of-basis equation \eqref{eq:basischangeExplicit} is the first example of transformation rules for quaternion tensors. Such transformations rely on the $\bbH\bbR$-multilinear form representation of quaternion tensors (\Cref*{def:quaternionTensor}) and the linearity properties of the underlying vector spaces. 
Extending the change-of-basis rule to arbitrary linear transformations in each variable leads to the definition of the $n$-mode product, which needs to be carefully specified for quaternion tensors due to $\bbH\bbR$-multilinearity.
From now on, we fix a basis for the $D$ vector spaces $\calS_1, \calS_2, \ldots, \calS_D$ and identify tensors with $D$-dimensional array of numbers.

\begin{definition}[$n$-mode product for quaternion tensors]\label{def:nModeProductQuaternionTensor}
	Let $\tens{T} \in \bbH^{N_1\times \cdots \times N_D}$ be a quaternion tensor of order $D$. 
	The $n$-mode-product operation corresponds to multiplying the $n$-mode fibers of $\tens{T}$ by a matrix $\bfU$ of appropriate size and domain, depending on the considered mode. 
	Explicitly, 
	\begin{itemize}
		\item the 1-mode product is defined as the direct quaternion matrix product by $\bfU \in \bbH^{J\times N_1}$ such that 
			  \begin{equation}
				  \left(\tens{T}\lnmod_1 \bfU\right)_{ji_2\ldots i_D} = \sum_{i_1=1}^{N_1}u_{ji_1}\tenselem{T}_{i_1i_2\ldots i_D},
			  \end{equation}
		\item for $1 < d < D$, the $d$-mode product is defined as the real matrix product by $\bfU \in \bbR^{J\times N_d}$ such that
		\begin{equation}
				\left(\tens{T}\nmod_d \bfU\right)_{i_1\ldots j \ldots i_D} = \sum_{i_d=1}^{N_d}\tenselem{T}_{i_1\ldots i_d\ldots i_D}u_{ji_d},
		\end{equation}
		\item the $D$-mode product is defined as the reverse quaternion matrix product by $\bfU \in \bbH^{J\times N_D}$ such that
			  \begin{equation}
				  \left(\tens{T}\rnmod_D \bfU\right)_{i_1\ldots i_{D_1}j} = \sum_{i_D=1}^{N_D}\tenselem{T}_{i_1i_2\ldots i_D}u_{ji_D}.
			  \end{equation}
	\end{itemize}

\end{definition}

Just like for quaternion matrix multiplication, symbols ${\rhd}$ and ${\lhd}$  indicate the type of quaternion matrix multiplication involved in mode-products when relevant.
In addition, one can express $n$-mode products in terms of unfoldings of quaternion tensors and quaternion matrix multiplication operations. 
Following standard practice, we denote by $\bfT_{(d)} \in \bbH^{N_d \times \prod_{i\neq d}N_i}$ the $d$-th mode unfolding of a quaternion tensor $\tens{T} \in \bbH^{N_1\times \cdots N_D}$ and use the indexing convention of \cite{kolda2009tensor}. 
The $n$-mode product operations relate to matrix unfoldings as
\begin{align}
	\tens{T}' = \tens{T}\lnmod_1 \bfU &\Longleftrightarrow \bfT_{(1)}' = \bfU \lmul \bfT_{(1)},\\
	\tens{T}' = \tens{T}\nmod_d \bfU &\Longleftrightarrow \bfT_{(d)}' = \bfU \bfT_{(d)},\\
	\tens{T}' = \tens{T}\rnmod_D \bfU &\Longleftrightarrow \bfT_{(D)}' = \bfU \rmul \bfT_{(D)},
\end{align}
i.e., the matrix unfoldings inherit the direct and reverse order of the $n$-mode product operation.

\Cref*{prop:commutativity} and \Cref*{prop:combination} show that the \cref{def:nModeProductQuaternionTensor} of $n$-modes product permits to preserve the classical properties of successive $n$-mode products: (i) commutativity between distinct modes and (ii) matrix composition between identical modes.
These desirable properties are directly inherited from the definition a quaternion tensor as a $\bbH\bbR$-multilinear form. 
\begin{proposition}[Commutativity]\label{prop:commutativity}
    Let $\tens{T} \in \bbH^{N_1 \times N_2 \times \cdots \times N_D}$ be a quaternion tensor of order $D$. 
	Let $\bfU_1 \in \bbH^{J_1\times N_1}$ and $\bfU_D \in \bbH^{J_D\times N_D}$ be arbitrary quaternion matrices, and let $\bfU_d \in \bbR^{J_d\times N_d}$ and $\bfU_{d'} \in \bbR^{J_{d'}\times N_{d'}}$ be arbitrary real matrices with $d\neq d', 1 < d, d' < D$.
    The following properties are satisfied:
    \begin{align}
        \left(\tens{T}\lnmod_1 \bfU_1 \right)\rnmod_D \bfU_D & = \left(\tens{T}\rnmod_D \bfU_D \right)\lnmod_1 \bfU_1 , \\
        \left(\tens{T}\lnmod_1 \bfU_1\right) \nmod_d \bfU_d   & = \left(\tens{T}\times_d\bfU_d\right)\lnmod_1\bfU_1 ,     \\
        \left(\tens{T}\nmod_d \bfU_d\right) \rnmod_D \bfU_D   & = \left(\tens{T}\rnmod_D \bfU_D \right)\nmod_d\bfU_d,\\
		\left(\tens{T}\nmod_d \bfU_d\right) \nmod_{d'} \bfU_{d'}   & = \left(\tens{T}\nmod_{d'} \bfU_{d'} \right)\nmod_d\bfU_d.
    \end{align}
\end{proposition}
\begin{proposition}[Composition]\label{prop:combination}
	Let $\tens{T} \in \bbH^{N_1 \times N_2 \times \cdots \times N_D}$ be a quaternion tensor of order $D$.
	Let $\bfU$ and $\bfV$ be matrices with adequate dimensions and values. 
    The following properties are satisfied:
	\begin{align}
		\tens{T}\lnmod_1 \bfU \lnmod_1 \bfV & =  \tens{T}\lnmod_1 (\bfV\lmul \bfU) , \\
		\tens{T}\rnmod_D \bfU \rnmod_D \bfV & =  \tens{T}\rnmod_D (\bfV\rmul \bfU) , \\
		\tens{T}\times_d \bfU \times_d \bfV     & = \tens{T}\times_d (\bfV\bfU), \quad 1 < d < D.
	\end{align}
\end{proposition}
Proofs of \Cref*{prop:commutativity} and \Cref*{prop:combination} are found by direct calculations.
\begin{remark}[Change-of-basis]
	The change of basis equation \eqref{eq:basischangeExplicit} can be rewritten as successive $n$-mode products. Let $\bfA^{(1)} \in \bbH^{N_1 \times N_1}, \bfA^{(2)} \in \bbR^{N_2\times N_2}, \ldots,  \bfA^{(D-1)} \in \bbR^{N_{D-1}\times N_{D-1}},  \bfA^{(D)} \in \bbH^{N_D\times N_D}$ the change-of-basis matrices. Then the change-of-basis equation can be rewritten as 
	\begin{equation}
		\tens{T} = \tens{\tilde{T}} \lnmod_1 \bfA^{(1)} \nmod_2 \bfA^{(2)} \cdots \nmod_{D-1} \bfA^{(D-1)}\rnmod_D  \bfA^{(D)},\label{eq:changeofbasisNmode}
	\end{equation}
	where $n$-mode products can be performed in any order according to \Cref{prop:commutativity}.
\end{remark}
\subsection{Tucker format of quaternion tensors}
\label{sub:tucker_format}
Without loss of generality, in the remainder of the paper we focus on the case of third order tensors. 
The next definition defines the Tucker format for quaternion tensors.

\begin{definition}[Tucker format]
Let $\tens{T} \in \bbH^{N_1 \times N_2\times N_3}$ be a third-order quaternion tensor. The tensor $\tens{T}$ is said to be expressed in Tucker format if there exists a core tensor $\tens{S} \in \bbH^{F_1\times F_2\times F_3}$ and matrices $\bfA \in \bbH^{N_1\times F_1}, \bfB \in \bbR^{N_2\times F_2}$ and $\bfC \in \bbH^{N_3\times F_3}$ such that 
\begin{equation}
	\tens{T} = \tens{S} \lnmod_1 \bfA \nmod_2 \bfB \rnmod_3 \bfC.
 \label{eq:QTucker}
\end{equation}
The Tucker format of $\tens{T}$ is written in short-hand notation as $\tens{T} =\llbracket\tens{S}; \bfA, \bfB, \bfC\rrbracket$.
\end{definition}
Direct calculations show that the entries of a tensor $\tens{T} = \llbracket\tens{S}; \bfA, \bfB, \bfC\rrbracket$ can be written in terms of the entries of the core tensor $\tens{S}$ and factor matrices $\bfA, \bfB, \bfC$ as 
\begin{equation}
	\tenselem{T}_{i_1i_2i_3} = \sum_{f_1= 1}^{F_1}\sum_{f_2= 1}^{F_2}\sum_{f_3= 1}^{F_3} A_{i_1f_1}B_{i_2f_2}S_{f_1f_2f_3}C_{i_3f_3},\label{eq:tuckerFormat_elementwise}
\end{equation}
for $i_1 = 1, \ldots, N_1, i_2 = 1, \ldots, N_2$ and $i_3 = 1, \ldots, N_3$. Note that due to quaternion non-commutativity, elements in \eqref{eq:tuckerFormat_elementwise} do not commute except for the entries of $\bfB$.
For tensors written in Tucker format it is also possible to derive expressions of their mode-unfoldings. 
Following the unfolding convention of Kolda and Bader \cite{kolda2009tensor}, if  $\tens{T} =\llbracket\tens{S}; \bfA, \bfB, \bfC\rrbracket$ then its matrix unfoldings along the first and last mode read
\begin{align}
	\bfT_{(1)} & = \bfA \lmul \bfS_{(1)}\lmul \left(\bfC \kron \bfB\right)^\top,\label{eq:unfoldingTucker_1}\\
	\bfT_{(3)} & = \bfC \rmul \bfS_{(3)}\rmul \left(\bfB \kron \bfA\right)^\top.\label{eq:unfoldingTucker_3}
\end{align}
However, due to quaternion non-commutativity, the unfolding along the central mode can only be expressed in a simple way when the core tensor is real-valued, i.e.,
\begin{equation}
	\bfT_{(2)} = \bfB \bfS_{(2)} \left(\bfC \rkron \bfA\right)^\top, \quad \text{when }\tens{S}\text{ is real-valued.}\label{eq:unfoldingTucker_2_real}
\end{equation}

\section{Quaternion Canonical Polyadic Decomposition}
\label{sec:quaternionCPD}

\subsection{Definition}
\label{sub:defQuaternionCPD}
	If in  \eqref{eq:QTucker} the tensor $\tens{S}$ is diagonal,  it enables the derivation of the quaternion equivalent of a well-known tensor model, the Canonical Polyadic Decomposition (CPD). By absorbing the diagonal entries of $\tens{S}$ in the factor matrices $\bfA, \bfB, \bfC$, we define the Quaternion Canonical Polyadic Decomposition (Q-CPD) as follows.

 \begin{definition}[Quaternion Canonical Polyadic Decomposition (Q-CPD)]
The Q-CPD of a  third-order quaternion tensor  $\tens{T} \in \bbH^{N_1 \times N_2\times N_3}$ is given element-wise by : 
\begin{equation}
	\tenselem{T}_{i_1i_2i_3} = \sum_{f= 1}^{F} A_{i_1f}B_{i_2f}C_{i_3f},\label{eq:CPDformat_elementwise}
\end{equation}
for $i_1 = 1, \ldots, N_1, i_2 = 1, \ldots, N_2$ and $i_3 = 1, \ldots, N_3$. The smallest value of $F$ for which  equality  \eqref{eq:CPDformat_elementwise} holds exactly is called the rank of $\tens{T}$.
By expressing the elements of the decomposition in matrix format,  the Q-CPD of $\tens{T}$ can be succinctly expressed  as $\tens{T} =\llbracket \bfA, \bfB, \bfC\rrbracket$, where  $\bfA \in \bbH^{N_1\times F}, \bfB \in \bbR^{N_2\times F}$,  $\bfC \in \bbH^{N_3\times F}.$ 
\end{definition}
\begin{remark}
	Contrary to quaternion matrices, the rank of a quaternion tensor is unambiguous as defined and thus the left or right distinction does not apply.
\end{remark}
Notation for the Q-CPD $\tens{T} =\llbracket \bfA, \bfB, \bfC\rrbracket$ refers to the fact that the Q-CPD can be interpreted as a Tucker model in which the core tensor can always be chosen to be the identity tensor. In particular, this enables a direct derivation of the unfoldings of the model; see Eqs. \eqref{eq:unfold1}-\eqref{eq:unfold3} below.
Moreover, similarly to the Tucker model, elements in \eqref{eq:CPDformat_elementwise} do not commute except for the entries of $\bfB$. First, let us analyze the trivial ambiguities of the Q-CPD.

Let $\tens{T} =\llbracket \bfA, \bfB, \bfC\rrbracket$ follow a CPD of rank $F$. 
The first trivial ambiguity corresponds to an arbitrary joint permutation of the columns of the factor matrices, i.e., $\tens{T} = \llbracket \bfA, \bfB, \bfC\rrbracket = \llbracket \bfA\boldsymbol{\Pi},  \bfB\boldsymbol{\Pi}, \bfC\boldsymbol{\Pi}\rrbracket$ with $\boldsymbol{\Pi} \in \bbR^{F\times F}$ an arbitrary permutation matrix. 
The second trivial ambiguity relates to scaling between factors: unlike the real and complex CPD, one has to distinguish between \emph{real} and \emph{quaternion} scaling ambiguities due to factor matrices taking values in different spaces. 
Define three diagonal scaling matrices $\boldsymbol{\Lambda}_{\bfA}, \boldsymbol{\Lambda}_{\bfC} \in \bbH^{F\times F}$ and $\boldsymbol{\Lambda}_{\bfB} \in \bbR^{F\times F}$ (with non-zero diagonal entries).
Direct calculations show that $\tens{T} = \llbracket \bfA, \bfB, \bfC\rrbracket = \llbracket \bfA\lmul\boldsymbol{\Lambda}_{\bfA},  \bfB\boldsymbol{\Lambda}_{\bfB}, \bfC\rmul \boldsymbol{\Lambda}_{\bfC}\rrbracket$ holds for arbitrary factor matrices $\bfA, \bfB, \bfC$ if and only if $(\boldsymbol{\Lambda}_{\bfA} \lmul \boldsymbol{\Lambda}_{\bfC}) \boldsymbol{\Lambda}_{\bfB} = \bfI_F$, where $\bfI_F$ is the $F\times F$ identity matrix. 
By exploiting the polar form of scaling coefficients further interpretation is possible. 
\begin{proposition}[Interpretation of scaling ambiguities]
	Consider three diagonal matrices  $\boldsymbol{\Lambda}_{\bfA} = \diag(\alpha_1, \ldots, \alpha_F) \in \bbH^{F\times F}$, $\boldsymbol{\Lambda}_{\bfB} = \diag(\beta_1, \ldots, \beta_F) \in \bbR^{F\times F}$ and $\boldsymbol{\Lambda}_{\bfC} = \diag(\gamma_1, \ldots, \gamma_F) \in \bbH^{F\times F}$. 
	Write quaternion entries in polar form such that $\alpha_f = \vert \alpha_f \vert \exp(\bmmu_{\alpha_f}\varphi_{\alpha_f})$ and $\gamma_f = \vert \gamma_f \vert \exp(\bmmu_{\gamma_f}\varphi_{\gamma_f})$ where $\vert \cdot \vert$ denotes the modulus and $(\bmmu, \varphi) $ the axis and angle of the quaternion, respectively. 
	Then, the scaling condition $(\boldsymbol{\Lambda}_{\bfA} \lmul \boldsymbol{\Lambda}_{\bfC}) \boldsymbol{\Lambda}_{\bfB} = \bfI_F$ can be equivalently interpreted as the combination of two ambiguities:
	\begin{itemize}
		\item a modulus scaling ambiguity: $\vert \alpha_f \vert\, \vert \beta_f\vert\,\vert \gamma_f\vert =1$, $f=1, \ldots, F$
		\item a angular scaling ambiguity: for $f=1, \ldots, F$, if \( \beta_f \) is positive, it implies there is a choice of \( \bmmu_{\alpha_f}, \varphi_{\alpha_f} \) and \( \bmmu_{\gamma_f}, \varphi_{\gamma_f} \), such that $\bmmu_{\alpha_f} = \bmmu_{\gamma_f}$ and $\varphi_{\alpha_f} + \varphi_{\gamma_f} = 0$.
        In the case \( \beta_f \) is negative, it implies there is a choice of \( \bmmu_{\alpha_f}, \varphi_{\alpha_f} \) and \( \bmmu_{\gamma_f}, \varphi_{\gamma_f} \), such that $\bmmu_{\alpha_f} = \bmmu_{\gamma_f}$ and $\varphi_{\alpha_f} + \varphi_{\gamma_f} = \pi$.
	\end{itemize}
	i.e., moduli must multiply to 1 and arguments of quaternion factors compensate each other according to the sign of $\beta_f$. 
\end{proposition}
\begin{proof}
	Starting from $(\boldsymbol{\Lambda}_{\bfA} \lmul \boldsymbol{\Lambda}_{\bfC}) \boldsymbol{\Lambda}_{\bfB} = \bfI_F$ we get $\alpha_f \beta_f \gamma_f =1$ for every $f=1, \ldots, F$. 
	Using polar forms, the latter reads $\vert \alpha_f \vert \exp(\bmmu_{\alpha_f}\varphi_{\alpha_f}) \beta_f \vert\gamma_f \vert \exp(\bmmu_{\gamma_f}\varphi_{\gamma_f}) = 1$. 
	Taking the modulus of this expression one gets $\vert \alpha_f \vert\, \vert \beta_f\vert\,\vert \gamma_f\vert =1$. 
	For the argument,  we have $\exp(\bmmu_{\alpha_f}\varphi_{\alpha_f}) \exp(\bmmu_{\gamma_f}\varphi_{\gamma_f}) = \frac{ \beta_f }{\vert  \beta_f \vert}$. By considering separately the cases $\beta_f > 0$ and $\beta_f < 0$ one obtains the desired property. 
\end{proof}

These ambiguities play a crucial role to define the uniqueness of the Q-CPD; see \cref*{sub:uniquenessQuaternionCPD} for more details. 
Finally, note that handling the Q-CPD trivial ambiguities does not pose any challenge as long as non-commutativity of quaternion multiplication is taken into account.

The matrix unfoldings of the quaternion CPD model follow directly from the unfoldings \eqref{eq:unfoldingTucker_1} -- \eqref{eq:unfoldingTucker_2_real} of the Tucker model.
As already mentioned, the Q-CPD can be viewed as a special case of the Tucker model with a core identity tensor; therefore the unfoldings of $\tens{T} =\llbracket \bfA, \bfB, \bfC\rrbracket$ read
\begin{align}
	\bfT_{(1)} & = \bfA \lmul \left(\bfC \krb \bfB\right)^\top, \label{eq:unfold1}\\
	\bfT_{(2)} &= \bfB \left(\bfC \rkrb \bfA\right)^\top,\label{eq:unfold2}\\
	\bfT_{(3)} & = \bfC \rmul \left(\bfB \krb \bfA\right)^\top.\label{eq:unfold3}
\end{align}
Once again, these matrix unfoldings are similar to those of the real and complex cases (see, e.g., \cite{kolda2009tensor}) yet matrix and Khatri-Rao products have direct or reverse orderings due to non-commutativity of the quaternion product. 
Moreover, these unfoldings are essential to the practical computation of the factor matrices $\bfA, \bfB$ and $\bfC$; see \Cref{sec:algo} for further details.

To conclude this section, note that it is also possible to derive mathematical expressions for the matrix slices of the Q-CPD \eqref{eq:CPDformat_elementwise}. 
The frontal slices $\bfT_{::i_3} \in \bbH^{N_1\times N_2}$, lateral slices $\bfT_{:i_2:} \in \bbH^{N_1\times N_3}$ and horizontal slices $\bfT_{i_1::} \in \bbH^{N_2\times N_3}$ of $\tens{T} =\llbracket \bfA, \bfB, \bfC\rrbracket$ can be written in terms of factor matrices as
\begin{align}
	\bfT_{i_1::} &= \bfB\left(\bfD^{(i_1)}_{\bfA} \lmul \bfC^\T\right), \quad \bfD^{(i_1)}_{\bfA} := \diag(\bfA_{i_1:}) \in \bbH^{F\times F}, \quad i_1=1, \ldots, N_1 \label{eq:horizontalSliceQCPD},\\
	\bfT_{:i_2:} &=\left(\bfA \bfD^{(i_2)}_{\bfB} \right)\lmul \bfC^\T, \quad \bfD^{(i_2)}_{\bfB} := \diag(\bfB_{i_2:}) \in \bbR^{F\times F}, \quad i_2=1, \ldots, N_2 \label{eq:lateralSliceQCPD},\\
	\bfT_{::i_3} &= \left(\bfA  \lmul \bfD^{(i_3)}_{\bfC} \right) \bfB^\T, \quad \bfD^{(i_3)}_{\bfC} := \diag(\bfC_{i_3:}) \in \bbH^{F\times F}, \quad i_3=1, \ldots, N_3 \label{eq:frontalSliceQCPD}.
\end{align}
These expressions can be deduced from the definition \eqref{eq:CPDformat_elementwise} of the Q-CPD.
They will be useful in the next section to study equivalent models of the Q-CPD. 

\subsection{Equivalent models}
\label{sub:equivModels_quaternionCPD}
In this subsection, we propose different ways to rewrite the 
Q-CPD as structured or constrained decompositions in the complex domain.
Two main approaches are investigated. 
The first one models complex adjoints of lateral slices of $\tens{T}$ while the second one models separately the two Cayley-Dickson parts of of $\tens{T}$. 
The former will allow us to establish uniqueness properties of the Q-CPD in  \Cref{sub:uniquenessQuaternionCPD}. 
The latter will be useful in \Cref{sec:algo} to derive a simple complex-domain Q-CPD algorithm.

\subsubsection{Complex adjoint representation}
\label{subsub:complexAdjointRep_quatCPD}

Consider a third-order tensor $\tens{T} \in \bbH^{N_1\times N_2 \times N_3}$ following a rank-$F$ quaternion CPD
such that $\tens{T} =\llbracket \bfA, \bfB, \bfC\rrbracket$.
Recall that the lateral slices of $\tens{T}$ are given by \eqref{eq:lateralSliceQCPD} and let us write $\bfA = \bfA_1 + \bfA_2 \bmj$ and $\bfC = \bfC_1 + \bfC_2 \bmj$, where $\bfA_1, \bfA_2 \in \bbC^{N_1\times F}_{\bmi}$ and $\bfC_1, \bfC_2 \in \bbC^{N_3\times F}_{\bmi}$.
Then, the direct complex adjoint of the $i_2$-th lateral slice \eqref{eq:lateralSliceQCPD} is given by
\begin{equation}
	\begin{split}
	\ladjC(\bfT_{:i_2:}) &= \ladjC\left((\bfA \lmul \bfD^{(i_2)}_{\bfB}) \lmul \bfC^\T\right)\\
	&=\ladjC(\bfA)\ladjC(\bfD^{(i_2)}_{\bfB})\ladjC(\bfC^\T)\\
	&= \begin{bmatrix}
		\bfA_1 & \bfA_2\\
		-\conj\bfA_2 & \conj{\bfA}_1
	\end{bmatrix} \begin{bmatrix}
		\diag(\bfB_{i_2:}) & 0\\
		0 & \diag(\bfB_{i_2:})
	\end{bmatrix}
	\begin{bmatrix}
		\bfC_1^\T & \bfC^\T_2\\
		-\bfC_2^\H & \bfC^\H_1
	\end{bmatrix}.\label{eq:latSlicesAugmentedTensor}
\end{split}
\end{equation}
Let us denote by $\tens{T}_{\bbC}$ the complex-valued tensor such that its frontal slices\footnote{Note that the permutation of the second mode (real) and the third mode (quaternionic) is for ease of presentation only and does not affect the generality of the equivalence.} are $\ladjC(\bfT_{:i_2:})$ for $i_2=1, \ldots, N_2$. 
It is a structured tensor of dimensions $2N_1 \times 2N_3 \times N_2$.
The expression \eqref{eq:latSlicesAugmentedTensor} of its (now) frontal slices shows that $\tens{T}_{\bbC}$ follows a standard complex CPD model with colinearities. 
Direct calculations yield, after reorganizing terms, the expression
\begin{equation}
	\tens{T}_{\bbC} = \sum_{f=1}^F \ladjC(\bfa_f)\ladjC(\bfc_f^\top)\circ \bfb_f = \sum_{f=1}^F \ladjC(\bfa_f)\radjC^\top(\bfc_f)\circ \bfb_f,
 	\label{eq:LL1_equivalentQCPD}
\end{equation}
where $\circ$ is the usual outer product.
Eq. \eqref{eq:LL1_equivalentQCPD} shows that $\tens{T}_{\bbC}$ follows a decomposition into $F$ rank-$(2, 2, 1)$ terms \cite{de_lathauwer_decompositions_2008}.
This result is summarized in the following proposition.
\begin{proposition}
Let $\tens{T} \in \bbH^{N_1\times N_2 \times N_3}$ be a third-order tensor as defined in \Cref{def:quaternionTensor} and denote by $\tens{T}_{\bbC} \in \bbCi^{2N_1 \times 2N_3 \times N_2}$ the tensor having frontal slices given by complex adjoints of lateral slices of $\tens{T}$.
Then $\tens{T}$ follows a rank-$F$ quaternion CPD if and only if $\tens{T}_{\bbC}$ follows a structured rank-(2,2,1) decomposition of the form \eqref{eq:LL1_equivalentQCPD}.
\label{prop:equivalenceLL1_quaternionCPD}
\end{proposition}
\begin{proof}
	The result holds by isomorphism between the set of quaternions matrices equipped with the direct quaternion matrix product and and the set of direct complex adjoints.
\end{proof}
\Cref*{prop:equivalenceLL1_quaternionCPD} is particularly useful to establish uniqueness results for the quaternion CPD; see \Cref{sub:uniquenessQuaternionCPD} below. 

It is worth noting that the ambiguities of the decomposition \eqref{eq:LL1_equivalentQCPD} are more restrictive than those of the standard unconstrained rank-$(L, L, 1)$ decomposition \cite{de_lathauwer_decompositions_2008}.
Indeed consider such an arbitrary decomposition of a complex tensor $\tens{T} = \sum_{f=1}^F \bfU_f\bfW^\T_f\circ\bfv_f$, where $\bfU_f, \bfW_f$ have full column rank $L$, and where $\bfv_f$ are vectors.
	Then it is well known that only the low-rank $\bfU_f\bfW^\T_f$ subspaces can be identified since 
	$\bfU_f\bfW^\T_f = \bfU_f \bfR_{\bfU} (\bfV_f\bfR_{\bfV})^\T$ where $\bfR_{\bfU}, \bfR_{\bfV}\in \bbC^{L\times L}$ are such that $\bfR_{\bfU}\bfR^\T_{\bfV} = \bfI_L$. 
For the structured rank-(2,2,1) decomposition \eqref{eq:LL1_equivalentQCPD}, the situation is more restrictive with the help of the following lemma. 
\begin{lemma}\label{lemma:LL1-ambiguities}
	Let $\bfp, \tilde{\bfp} \in \bbH^M$ and $\bfq, \tilde{\bfq}\in \bbH^N$.
	Then $\ladjC(\bfp)\radjC^\top(\bfq) = \ladjC(\tilde{\bfp})\radjC^\top(\tilde{\bfq})$ if and only if there exists a non-zero quaternion $r \in \bbH$ such that $\tilde{\bfp} = \bfp r $ and $ \tilde{\bfq} = r^{-1} \bfq$.
\end{lemma}
\begin{proof}
		Using complex adjoint properties, one has $\ladjC(\bfp)\radjC^\top(\bfq) =\ladjC(\bfp)\ladjC(\bfq^\top)= \ladjC(\bfp\bfq^\top)$. Thus $\ladjC(\bfp)\radjC^\top(\bfq) = \ladjC(\tilde{\bfp})\radjC^\top(\tilde{\bfq})$ if and only if $\ladjC(\bfp\bfq^\top) = \ladjC(\tilde{\bfp}\tilde{\bfq}^\top)$. By the properties of the complex adjoint, this is equivalent to saying that $\bfp \bfq^\top = \tilde{\bfp}\tilde{\bfq}^\top$.
		The equality is satisfied if and only there exists a non-zero $r\in \bbH$ such that $\tilde{\bfp} = \bfp r$ and $\tilde{\bfq}^\top = r^{-1}\tilde{\bfq}^\top$. Since $r$ is a scalar, this equivalently means that $\tilde{\bfq} = r^{-1}\bfq$ and it concludes the proof.  
	\end{proof}
	Applying \Cref{lemma:LL1-ambiguities} to the decomposition \eqref{eq:LL1_equivalentQCPD}, one obtains that  $\ladjC(\bfa_f)\radjC^\top(\bfc_f) = \ladjC(\bfa_f)\bfR_{\bfA}(\radjC(\bfc_f)\bfR_{\bfC})^\top$ if and only if there exists a non-zero scalar $r_{f} \in \bbH$ such that $\bfR_{\bfA} = \ladjC(r_{f})$ and  $\bfR_{\bfC}^\top= \ladjC(r_{f}^{-1})$, or equivalently, $\bfR_{\bfC}= \ladjC^\top(r_{f}^{-1}) = \radjC(r_f^{-1})$.
	In other terms, the ambiguities of the rank-(2,2,1) factors are directly inherited from those of the quaternion factors of the Q-CPD. 

\subsubsection{Coupled CONFAC model}\label{subsub:secCONFAC}
Suppose again that $\tens{T} =\llbracket \bfA, \bfB, \bfC\rrbracket$ follows a rank-$F$ Q-CPD and consider its Cayley-Dickson decomposition as $\tens{T}=\tens{T}_1+\tens{T}_2\bmj$, where $\tens{T}_1, \tens{T}_2 \in \bbC_{\bmi}^{N_1\times N_2\times N_3}$.
Then, another approach to build a complex representation of the Q-CPD of $\tens{T}$ consists in modeling separately $\tens{T}_1$ and $\tens{T}_2$ and interpreting the Q-CPD as a coupled decomposition of these two components. 
 In the following we define as a constraint matrix any matrix that is full row rank and whose each column have one, and only one, non-null entry.
 We then say that a third order tensor $\tens{X} \in \bbC^{N_1\times N_2\times N_3}_{\bmi}$ admits a CONFAC decomposition \cite{favier2014overview} (or follows a CONFAC model) of rank $F_c$ if there exist three matrices,
$\bfU, \bfV$ and $\bfW$ of sizes $(N_1,L_1),(N_2,L_2),(N_3,L_3)$ respectively and three constraint matrices, $\bPsi$, $\bPhi$, $\bOm$ of sizes $(L_1,F_c),(L_2,F_c),(L_3,F_c)$ such that 
\begin{equation}\label{confac1}
		\tens{X}=\ldbrack \bfU\bPsi, \bfV\bPhi, \bfW\bOm \rdbrack.
	\end{equation}
Provided that the constraint matrices $\mathbf{\Psi}$, $\mathbf{\Phi}$, $\mathbf{\Omega}$ are known,  the factor matrices $\bfU, \bfV$ and $\bfW$ are easily identified by means of an ALS algorithm \cite{deAlmeida2008confac}.
Note that the CONFAC model is closely related to the PARALIND model \cite{Paralind}. \cref{theoconfacc} states that the Q-CPD implies that each component follows a CONFAC decomposition model.
	\begin{theorem}[Link between the Q-CPD and the complex CONFAC decomposition]\label{theoconfacc}
		Let $\tens{T}$ be a third order tensor of quaternions of size $(N_1,N_2,N_3)$. 
  Let $\ldbrack \bfA, \bfB, \bfC \rdbrack$ be its rank-$F$ Q-CPD and $\tens{T}_1$ and $\tens{T}_2$ the two $\bbCi$-valued tensors such that $\tens{T}=\tens{T}_1+\tens{T}_2\bmj$.
  Then, $\tens{T}_1$ and $\tens{T}_2$ each admits a CONFAC decomposition of rank $2F$ given by 
\begin{subequations}
	\begin{align}
		\tens{T}_1=&\ldbrack  \bfU\bPsi, \bfB\bPhi,\widebar{\bfW}\bfI_{2F}\rdbrack \label{cconfac1},\\
		\tens{T}_2=&\ldbrack  \bfU\bfI_{2F}, \bfB\bPhi,\bfW\bfI_{2F} \rdbrack \label{cconfac2},
	  \end{align}
\end{subequations}
  where  the factor matrices $\bfU$ and $\bfW$ read in terms of Cayley-Dickson factors of $\bfA$ and $\bfC$
  \begin{equation}
      \bfU= \begin{bmatrix}
				\bfA_1&\bfA_2\end{bmatrix},\quad \bfW=
			\begin{bmatrix}
				\bfC_2&\widebar{\bfC}_1 
			\end{bmatrix},
  \end{equation}
  and where the constraint matrices $\mathbf{\Psi}$ and $\mathbf{\Phi}$ are given by
  \begin{equation}\label{eq:constraint_matrices_confac}
      \mathbf{\Psi}=
			\begin{bmatrix}
				0&\bfI_F\\ -\bfI_F&0 
			\end{bmatrix},\quad  \mathbf{\Phi}=
			\begin{bmatrix}
				\bfI_F&\bfI_F
			\end{bmatrix}.
  \end{equation} 
	\end{theorem}		

	\begin{proof}
		Consider the Cayley-Dickson decomposition of quaternion factor matrices as $\bfA = \bfA_1 + \bfA_2 \bmj$ and $\bfC = \bfC_1 + \bfC_2 \bmj$.
		Then, the Q-CPD of $\tens{T} = \llbracket \bfA, \bfB, \bfC\rrbracket$ rewrites as 
		\[
		\tens{T}=\sum_{f=1}^F	(\bfa_{1f}+\bfa_{2f}\bmj)\circ \bfb_{f}\circ(\bfc_{1f}+\bfc_{2f}\bmj)
		\]
		where $\bfa_{1f}$ is the column $f$ of matrix $\bfA_1$. This leads to
		\begin{align}
			\tens{T}_1=&\sum_{f=1}^F  \bfa_{1f}\circ\bfb_{f}\circ\bfc_{1f}-
			\bfa_{2f}\circ\bfb_{f}\circ\widebar{\bfc}_{2f},\label{eqT1c}\\
			\tens{T}_2=&\sum_{f=1}^F \bfa_{1f}\circ\bfb_{f}\circ\bfc_{2f}+
			\bfa_{2f}\circ\bfb_{f}\circ\widebar{\bfc}_{1f}.\label{eqT2c}
		\end{align}		
Thus, equation \eqref{eqT1c} decomposes $\tens{T}_1$ as a sum of two CPDs and there is an unfolding of $\tens{T}_1$ of size $(N_2, N_3N_1)$, denoted $\bfT_{1}^B$ and such that
\begin{equation}
	\begin{split}
				\bfT_{1}^B&=\bfB((-\bfA_2\krb\widebar{\bfC}_2)^\T+(\bfA_1\krb\bfC_1)^\T)\\
&=\begin{bmatrix}
				\bfB&\bfB\end{bmatrix}		\begin{bmatrix}
					-\bfA_2\krb\widebar{\bfC}_2&\bfA_1\krb\bfC_1
				\end{bmatrix}^\T\\   
			&=	\begin{bmatrix}
			\bfB&\bfB\end{bmatrix}		\begin{bmatrix}\begin{bmatrix}-\bfA_2& \bfA_1\end{bmatrix}\krb\begin{bmatrix}\widebar{\bfC}_2 &\bfC_1\end{bmatrix}	\end{bmatrix}^\T. 
	\end{split}\label{eqTB1}	
    \end{equation}
Introducing $\bfU= [\bfA_1\:\bfA_2]$ and $\bfW=[\bfC_2\:\widebar{\bfC}_1]$, Equation  \eqref{eqTB1} rewrites
    \begin{equation}\label{eqT0}
			\bfT_{1}^B=\bfB\mathbf{\Phi}	(\bfU\mathbf{\Psi}\krb\widebar{\bfW})^\T 
		\end{equation}
  which is equivalent to equation \eqref{cconfac1}. 
The same reasoning holds for $\tens{T}_2$ and leads to
    \begin{equation}\label{eqT1}
			\bfT_{2}^B=\bfB\mathbf{\Phi}	(\bfU\krb\bfW)^\T 
		\end{equation}
  which is equivalent to equation \eqref{cconfac2}.
		Eventually, one can immediately check that $\mathbf{\Psi}$, $\mathbf{\Phi}$, and $\bfI_{2F}$ have full row rank and their columns have one and only one non-null entry.
	\end{proof}
 
 The system of equations \eqref{cconfac1} and \eqref{cconfac2} (or \eqref{eqT0} and \eqref{eqT1}) defines a coupled CONFAC decomposition\footnote{One can easily check that a similar reasoning applies to the four real components of $\tens{T}$ leading to a real coupled CONFAC decomposition of the four tensors.}. 
 It is worth mentioning that this decomposition was introduced by \cite[section IV.A]{deAlmeida2012cafconfac} in a totally different context.
The main advantage of this model with respect to the complex adjoint representation of equation \eqref{eq:LL1_equivalentQCPD} is that it only involves the $2N_1N_2N_3$ complex equations that are directly induced by the Q-CPD model.
It is thus more suitable for calculating 
the Q-CPD in the complex field.
Moreover, it has been shown in \cite{deAlmeida2012cafconfac} that it can be computed using a simple algorithm, from which we have drawn inspiration for the Q-CPD algorithm presented in Section \ref{qcalgo}.

In addition, we will show in \Cref{sub:uniquenessQuaternionCPD} that under some assumptions, the real factor matrix can be identified from any of the two CONFAC models involved in the coupled CONFAC decomposition.

\subsection{Uniqueness of the Q-CPD}\label{sub:uniquenessQuaternionCPD}
As explained in \Cref{sub:defQuaternionCPD}, the Q-CPD exhibits trivial scaling and permutation ambiguities, leading to the following definition of uniqueness. 

\begin{definition}[Uniqueness of the Q-CPD]\label{def:uniquenessQCPD}
	Let $\tens{T} = \llbracket \bfA, \bfB, \bfC\rrbracket$ admit a rank-$F$ Q-CPD. We say that the Q-CPD is essentially unique (or simply, unique) if the only factor matrices $\tilde{\bfA}, \tilde{\bfB}, \tilde{\bfC}$ such that $\tens{T} = \llbracket \tilde{\bfA}, \tilde{\bfB}, \tilde{\bfC}\rrbracket$ are related through
	\begin{equation}
		\tilde{\bfA} = \bfA \lmul \left(\boldsymbol{\Lambda}_{\bfA}\boldsymbol{\Pi}\right),\quad \tilde{\bfB} = \bfB\boldsymbol{\Lambda}_{\bfB}\boldsymbol{\Pi},\quad \tilde{\bfC}=\bfC\rmul \left(\boldsymbol{\Lambda}_{\bfC}\boldsymbol{\Pi}\right),
	\end{equation}
	where $\boldsymbol{\Pi} \in \bbR^{F\times F}$ is a permutation matrix, $\boldsymbol{\Lambda}_{\bfA}, \boldsymbol{\Lambda}_{\bfC}$ are quaternion-valued diagonal scaling matrices and where $\boldsymbol{\Lambda}_{\bfB}$ is a real diagonal matrix such that $(\boldsymbol{\Lambda}_{\bfA} \lmul \boldsymbol{\Lambda}_{\bfC}) \boldsymbol{\Lambda}_{\bfB} = \bfI_F$. 
\end{definition}

Now, uniqueness results for the quaternion CPD can be readily obtained through the equivalence with the structured rank-$(2,2,1)$ decomposition \eqref{eq:LL1_equivalentQCPD} (see also \Cref{prop:equivalenceLL1_quaternionCPD}).
This statement is made precise in the following proposition, which is a direct consequence of \Cref*{prop:equivalenceLL1_quaternionCPD}.
\begin{proposition}
	Let $\tens{T} \in \bbH^{N_1\times N_2 \times N_3}$ be a third-order tensor and denote by $\tens{T}_{\bbC} \in \bbCi^{2N_1 \times N_2 \times 2N_3}$ the tensor constructed from direct complex adjoints of lateral slices of $\tens{T}$. 
	Suppose that $\tens{T}$ follows a rank-$F$ quaternion CPD \eqref{eq:CPDformat_elementwise}. Then the quaternion CPD is unique if and only if the structured rank-(2,2,1) decomposition \eqref{eq:LL1_equivalentQCPD} of $\tens{T}_{\bbC}$ is unique.
	\label{prop:uniquenessEquiv}
\end{proposition}
The proof is straightforward by observing that the mapping between the factor matrices of $\tens{T}$ and those of $\tens{T}_{\bbC}$ is one-to-one. 
Building on this result, \Cref{thm:uniquenessQCPD_kruskal} below establishes a sufficient uniqueness condition for the Q-CPD. 
\begin{theorem}
	\label{thm:uniquenessQCPD_kruskal}
	Let $\bfA \in \bbH^{N_1\times F}, \bfB \in \bbR^{N_2\times F}$ and $\bfC \in \bbH^{N_3\times F}$ be factor matrices. 
	Suppose that (at least) one of the following conditions is satisfied:
	\begin{enumerate}
		\item $\rrank{\bfA} = F$, $\lrank{\bfC} = F$ and $\bfB$ does not have proportional columns;
		\item $\krank(\bfB) =F$ and $\rkrank(\bfA) + \lkrank(\bfC) \geq F+2$;
		\item $\rkrank(\bfA) = F$ and $\krank(\bfB) +\lkrank(\bfC) \geq F+2$;
		\item $\lkrank(\bfC) = F$ and $\rkrank(\bfA) + \krank(\bfB) \geq F+2$;
		\item $N_1N_3 \geq F$ and $\rkrank(\bfA) + \krank(\bfB) + \lkrank(\bfC) \geq 2F+2$.
	\end{enumerate}
	Then the Q-CPD $\tens{T} = \llbracket \bfA, \bfB, \bfC\rrbracket$ is unique in the sense of \Cref{def:uniquenessQCPD}.
\end{theorem}
\begin{proof}
	The sufficient uniqueness conditions for the Q-CPD derive directly from the usual uniqueness conditions for rank-$(L, L, 1)$ decompositions \cite{de_lathauwer_decompositions_2008}, using the equivalence of \Cref{subsub:complexAdjointRep_quatCPD} and our \Cref*{lemma:kprank_krank_quat}.
	Let $\tens{T} = \llbracket \bfA, \bfB, \bfC\rrbracket$ and construct the complex adjoint tensor $\tens{T}_{\bbC}$ as in \eqref{eq:LL1_equivalentQCPD}. 
	From \Cref{prop:uniquenessEquiv} the Q-CPD of $\tens{T}$ is unique if and only if the structured rank-(2,2,1) decomposition of $\tens{T}_{\bbC}$ is unique.
	Following the approach in \cite{de_lathauwer_decompositions_2008}, remark that the horizontal stack of rank-$2$ factors simply corresponds to the direct (resp. reverse) column-wise complex adjoint \eqref{eq:definition_adjC_permuted} of the matrix $\bfA$ (resp. $\bfC$), i.e., 
	\begin{align*}
		\ladjCp(\bfA) = \begin{bmatrix}
			\ladjC(\bfa_1) & \cdots & \ladjC(\bfa_F)
		\end{bmatrix} \in \bbH^{2N_1 \times 2F},\\ 		\radjCp(\bfC) = \begin{bmatrix}
			\radjC(\bfc_1) & \cdots & \radjC(\bfc_F)
		\end{bmatrix} \in \bbH^{2N_3 \times 2F}.
	\end{align*} 
	Recall from \Cref{remark:permuted_adjc} that there exists a permutation matrix $\boldsymbol{\Pi} \in \bbR^{2F\times 2F}$ such that $\ladjCp(\bfA) = \ladjC(\bfA) \boldsymbol{\Pi}$ and $\radjCp(\bfC) = \radjC(\bfC) \boldsymbol{\Pi}$. 
	Therefore uniqueness of the rank-$(2,2,1)$ decomposition of $\tens{T}_{\bbC}$ can directly be characterized in terms of matrices $\ladjC(\bfA), \radjC(\bfC)$ and $\bfB$. 
	Moreover, from the properties of the complex adjoints we have that $\rank{\ladjC(\bfA)} =\rank{\ladjCp(\bfA)}=2F$ if and only if $\rrank{\bfA} = F$ (\Cref{prop:left_right_ranks_complex_adj}) and $\kprank(\ladjCp(\bfA))= \rkrank(\bfA)$ (by \Cref{lemma:kprank_krank_quat}). 
	Similarly for $\bfC$ one has $\rank{\radjC(\bfC)} =\rank{\radjCp(\bfC)}=2F$ if and only if $\lrank{\bfC} = F$ and $\kprank(\radjCp(\bfC))= \lkrank(\bfC)$.
	It now suffices  to apply, in order, Theorems 4.1, 4.4, 4.5(a), 4.5(b) and 4.7 \cite{de_lathauwer_decompositions_2008} to the rank-$(2,2,1)$ decomposition of $\tens{T}_{\bbC}$ to obtain conditions 1., 2., 3., 4., and 5.
\end{proof}
\begin{remark}
Condition 5 in \Cref{thm:uniquenessQCPD_kruskal} resembles the usual Kruskal uniqueness condition for the CPD in the real and complex cases, up to \emph{(i)} the necessary distinction between left and right Kruskal ranks of a quaternion matrix; and \emph{(ii) }the (usually non-restrictive) condition $N_1N_3 \geq F$, i.e., that the product of the dimensions of the two quaternion-valued factors is at least equal to the rank of the tensor. In particular, when $F \leq \max (N_1, N_3)$, the condition is automatically satisfied and Condition 5 boils down to the well-known Kruskal condition generalized to the quaternion case. 
\end{remark}
\begin{remark}
	For a generic quaternion matrix, left and right $\krank$-ranks are both equal, i.e., if $\bfA \in \bbH^{M\times F}$ is generic, then $\lkrank(\bfA) = \rkrank(\bfA) = \min(M, F)$. This results directly from the definition of complex matrix adjoints \eqref{eq:left_complexAdjointMatrix} -- \eqref{eq:right_complexAdjointMatrix} and \Cref{lemma:kprank_krank_quat}.
\end{remark}

Eventually, \cref{prop1} states that the CONFAC decomposition of only one component of the Cayley-Dickson decomposition of  $\tens{T}$ is likely to be sufficient to identify the real factor matrix of the Q-CPD but insufficient to identify the quaternion factor matrices.
	\begin{proposition}[Essential uniqueness of $\bfB$]\label{prop1}
 Let $\tens{T}$ be a third order tensor of quaternions such that $\tens{T}=\tens{T}_1+\tens{T}_2\bmj$. 
 Let $\ldbrack \bfA, \bfB, \bfC \rdbrack$ be the rank-$F$ Q-CPD of $\tens{T}$ 
		and $\ldbrack \bfU\bPsi, \bfB\bPhi, \widebar{\bfW}\bfI_{2F} \rdbrack$ the CONFAC decomposition of $\tens{T}_1$ given in  \cref{theoconfacc}.
  If $\bfB$, $\bfU$ and $\bfW$ are full column rank
  then $\bfB$ is essentially unique. 
	\end{proposition}	
A proof is given in the Appendix. This result  also holds for  $\tens{T}_2$ and Equation \eqref{cconfac2}.

	\section{Algorithms for computing the quaternion CPD}\label{sec:algo}

Given an arbitrary tensor $\tens{T} \in \bbH^{N_1\times N_2\times N_3}$, the Q-CPD approximation problem of rank-$F$ can be formulated as the following optimization problem
\begin{equation}
	(\hat{\bfA}, \hat{\bfB}, \hat{\bfC}) = \argmin_{\bfA, \bfB, \bfC} \left\Vert \tens{T}- \llbracket \bfA, \bfB, \bfC\rrbracket\right\Vert_F^2.\label{eq:qcpd-approx}
\end{equation}
To solve \eqref{eq:qcpd-approx}, we introduce two algorithms that operate in the quaternion domain and the complex domain, respectively. 
Both algorithms rely on the well known Alternating Least Squares procedure (ALS). 
As we will show, they yield equivalent update rules; however subtle differences arise in their practical implementation, as discussed in \Cref{sec:numerical_simulations}.

	\subsection{Computing the Q-CPD in $\mathbb{H}$: the Q-ALS algorithm}

The Q-ALS algorithm extends the ALS algorithm to the Q-CPD.
It resorts to the matrix unfoldings of the Q-CPD given in equations \eqref{eq:unfold1}-\eqref{eq:unfold3}.
Thereby, at each iteration the factor matrices are updated as:
\begin{align}
		\bfA &\leftarrow \argmin_{\bfA} \;\; \lVert \bfT_{(1)} - \bfA \lmul (\bfC \krb \bfB)^\top\rVert^2_F,\label{minAh}\\
		\bfB  &\leftarrow \argmin_{\bfB} \;\; \lVert \bfT_{(2)} - \bfB  (\bfC \rkrb \bfA)^\top\rVert^2_F \ \  \text{s.t.} \ \ \bfB\text{ has real-valued entries}, \label{minBh} \\
		\bfC  &\leftarrow \argmin_{\bfC} \;\; \lVert \bfT_{(3)} - \bfC \rmul (\bfB \krb \bfA)^\top\rVert^2_F,\label{minCh}
\end{align}
where $\Vert \cdot \Vert_F$ denotes the Frobenius norm of a matrix.
We thus have to solve three different kinds of quaternion matrix least squares subproblems:
\begin{align}
    \hat{\bfX}_{\scriptscriptstyle \rhd} &= \argmin_{\bfX} \lVert \bfM - \bfX\lmul \bfN\rVert^2_F\label{eq:LSl} ,\\
    \hat{\bfX} &= \argmin_{\bfX} \lVert \bfM - \bfX \bfN\rVert^2_F \ \  \text{s.t.} \ \ \bfX\text{ has real-valued entries}  ,\label{eq:LS}\\
    \hat{\bfX}_{\scriptscriptstyle \lhd} &=\argmin_{\bfX} \lVert \bfM - \bfX\rmul \bfN\rVert^2_F\label{eq:LSr} .
\end{align}

The solution to \eqref{eq:LSl} can be found in \cite{xu_theory_2015} and is given by   $ \hat{\bfX}_{\scriptscriptstyle \rhd} = \bfM \lmul \bfN^\H \lmul \left(\bfN\lmul \bfN^\H\right)^{-1}\:.$
(Note that in this section, quaternion matrix inverses are always defined with respect to the direct quaternion matrix product and can be computed through inversion of the corresponding complex adjoint, see \cite[Theorem 4.2]{zhang_quaternions_1997}.)
The second subproblem is a constrained optimization problem since $\bfX$ is constrained to belong to the set of real-valued matrices. The optimal solution can be found in \cite[Appendix B]{flamant_quaternion_2020} and is given by 
$ \hat{\bfX} = \re{\bfM \lmul \bfN^\H} \re{ \left(\bfN\lmul \bfN^\H\right)}^{-1}\:.$
Eventually, in order to solve \eqref{eq:LSr}, we resort to the properties of direct and reverse quaternion matrix products: $\lVert \bfM - \bfX\rmul \bfN\rVert^2_F = \lVert \bfM^\T - \bfN^\T \lmul \bfX^\T\rVert^2_F$ so that the solution is given by: $ \hat{\bfX}_{\scriptscriptstyle \lhd} = \bfM \rmul \bfN^\H \rmul \left[\left(\conj{\bfN}\lmul \bfN^\T\right)^{-1}\right]^\T\:.$

As a consequence, the updates computed during one iteration of the Q-ALS algorithm are given by:
\begin{align}
 \bfA&\leftarrow\bfT_{(1)}\lmul (\conj{\bfC \krb \bfB}) \lmul \left((\bfC \krb \bfB)^\T\lmul (\conj{\bfC \krb \bfB})\right)^{-1} \label{Aupdh}\\
 \bfB&\leftarrow\re{\bfT_{(2)}\lmul (\conj{\bfC \rkrb \bfA})} \re{ (\bfC \rkrb \bfA)^\T\lmul (\conj{\bfC \rkrb\bfA})}^{-1}\label{Bupdh}\\
\bfC&\leftarrow\bfT_{(3)}\rmul (\conj{\bfB \krb \bfA}) \rmul \left(\left((\bfB \krb \bfA)^\H\lmul (\bfB \krb \bfA)\right)^{-1}\right)^\T.\label{Cupdh}
\end{align}
The Q-ALS updates resemble those of the usual ALS updates for computing the real and complex CPD. Therefore we anticipate that existing convergence results for ALS \cite{uschmajew2012local} can be adapted to the quaternion setting, exploiting recent advances in quaternion optimization \cite{flamant2021general,xu_enabling_2015,xu_theory_2015}. 
However, a precise characterization of such properties requires a dedicated study, which is left for future work.

\subsection{Computing the Q-CPD in $\mathbb{C}$: the C-ALS algorithm}\label{qcalgo}
This second approach relies on the complex coupled CONFAC representation of the Q-CPD presented in \Cref{subsub:secCONFAC}.
First, start from the Cayley-Dickson decomposition of $\tens{T} = \tens{T}_1 + \tens{T}_2 \bmj$ and recall that, if $\tens{T}$ admits a Q-CPD of rank-$F$, then from \Cref{theoconfacc} one has $\tens{T}_1 =\ldbrack  \bfU\bPsi, \bfB\bPhi,\widebar{\bfW}\rdbrack$ and $\tens{T}_2=\ldbrack  \bfU, \bfB\bPhi,\bfW \rdbrack$, where $\bfU= [\bfA_1\:\bfA_2]$ and $\bfW=[\bfC_2\:\widebar{\bfC}_1]$ are defined from Cayley-Dickson parts of $\bfA$ and $\bfC$, respectively, and where the constraint matrices $\bPsi$ and $\bPhi$ are given in \eqref{eq:constraint_matrices_confac}. 
Then one can rewrite the optimization problem \eqref{eq:qcpd-approx} as 
\begin{equation}
	(\hat{\bfU}, \hat{\bfB}, \hat{\bfW}) = \argmin_{\bfU, \bfB, \bfW} \left\Vert \tens{T}_1 - \ldbrack  \bfU\bPsi, \bfB\bPhi,\widebar{\bfW}\rdbrack\right\Vert_F^2 + \left\Vert\tens{T}_2 -  \ldbrack  \bfU, \bfB\bPhi,\bfW \rdbrack\right\Vert_F^2 .\label{eq:qcpd-approx_confac}
\end{equation}
since the Frobenius norm is preserved by the Cayley-Dickson decomposition. 
The structure of the cost function in \eqref{eq:qcpd-approx_confac} enables the derivation of ALS updates for $\bfU, \bfB$ and $\bfW$, leading to the C-ALS algorithm detailed below.

For simplicity, let us start with the estimation of the real matrix $\bfB$ given complex matrices $\bfU$ and $\bfW$. 
Denote by $\bfT_{1}^B$ and $\bfT_{2}^B$ the unfoldings of $\tens{T}_1$ and $\tens{T}_2$ along the second mode, respectively and define $\bfT^B=\begin{bmatrix}
		\bfT_{2}^B&\widebar{\bfT}_{1}^B
	\end{bmatrix}$.
	Then, exploiting unfoldings \eqref{eqT0} and \eqref{eqT1} of the coupled CONFAC model along the second mode, the $\bfB$-update optimization problem can be written as 
\begin{equation}\label{minB}
\bfB  \leftarrow \argmin_{\bfB} \;\; \lVert \bfT^B - \bfB\mathbf{\Phi}(\ladjC(\bfA)\krb \bfW)^\T\rVert^2_F \ \  \text{s.t.} \ \ \bfB\text{ has real-valued entries},
\end{equation}
which arises from the definition of $\bfU$ as a horizontal stack of Cayley-Dickson parts of $\bfA$ and simple manipulations of the unfoldings \eqref{eqT0} and \eqref{eqT1}.
The explicit solution to \eqref{minB} is given below in \eqref{Bupd}.

A similar approach is used to get the $\bfU$ and $\bfW$-updates. 
For $\bfW$, define a matrix $\bfT^C=	\begin{bmatrix}
		\bfT_{2}^C&\widebar{\bfT}_{1}^C
	\end{bmatrix}$ where $\bfT_{1}^C$ and $\bfT_{2}^C$ are the unfoldings along the third mode of $\tens{T}_1$ and $\tens{T}_2$, respectively. 
	From unfoldings of the CONFAC model, the $\bfW$-update is found as the solution to the optimization problem 
	\begin{equation}\label{minC}
\bfW  \leftarrow \argmin_{\bfW} \;\; \lVert \bfT^C - \bfW(\ladjC(\bfA)\krb \bfB\mathbf{\Phi})^\T\rVert^2_F.
\end{equation}
For $\bfU$, similar but slightly more tedious manipulations yield to 
\begin{equation}\label{minA}
\bfU  \leftarrow \argmin_{\bfU} \;\; \lVert \bfT^A - \bfU(\radjC(\bfC)\krb \bfB\mathbf{\Phi})^\T\rVert^2_F,
\end{equation}
where this time $\bfT^A = \begin{bmatrix}
		\bfT_{1}^A&\bfT_{2}^A
	\end{bmatrix}$ gathers the unfoldings of $\tens{T}_1$ and $\tens{T}_2$ along the first mode.
Eventually, the C-ALS procedure solves the Q-CPD approximation problem of rank-$F$ \eqref{eq:qcpd-approx} thanks to the coupled CONFAC interpretation \eqref{eq:qcpd-approx_confac} by iterating the following explicit updates 
\begin{align}
	\label{Uupd}
\bfU\leftarrow \bfT^A (\conj{\radjC(\bfC)\krb \bfB\mathbf{\Phi}}) \left( (\radjC(\bfC)\krb \bfB\mathbf{\Phi})^\T  (\conj{\radjC(\bfC)\krb \bfB\mathbf{\Phi}})\right)^{-1},\\
\label{Bupd}
\bfB\leftarrow\re{\bfT^B (\conj{\ladjC(\bfA) \krb \bfW})\mathbf{\Phi}^\T} \re{ \mathbf{\Phi}\left((\ladjC(\bfA) \krb \bfW)^\T  (\conj{\ladjC(\bfA) \krb\bfW})\right)\mathbf{\Phi}^\T}^{-1},\\
\label{Wupd}
\bfW\leftarrow \bfT^C (\conj{\ladjC(\bfA)\krb \bfB\mathbf{\Phi}}) \left( (\ladjC(\bfA)\krb \bfB\mathbf{\Phi})^\T  (\conj{\ladjC(\bfA)\krb \bfB\mathbf{\Phi}})\right)^{-1}.
\end{align}
At convergence, matrices $\bfA$  and $\bfC$ are directly deduced from $\bfU$ and $\bfW$.

It is worth mentioning that, by definition of matrices $\bfU$ and $\bfW$ on the one hand, and by construction of the C-ALS cost function \eqref{eq:qcpd-approx_confac} on the other hand, the least squares problems in \eqref{minB}, \eqref{minC} and \eqref{minA} are equivalent to those in \eqref{minBh}, \eqref{minCh} and \eqref{minAh}, respectively. 
As a result, Q-ALS iterations \eqref{Aupdh}-- \eqref{Cupdh} and C-ALS iterations \eqref{Uupd}--\eqref{Wupd} are one-to-one as well. 
In other words, the Q-ALS and C-ALS algorithms provide equivalent algorithms to solve the Q-CPD approximation problem \eqref{eq:qcpd-approx} of rank-$F$, operating in the quaternion and complex domain, respectively. 
Q-ALS updates enjoy somewhat simpler explicit expressions by leveraging directly the Q-CPD model; however practical implementation requires the use of dedicated quaternion linear algebra librairies. 
In contrary, the C-ALS algorithm can directly take advantage of standard (complex) linear algebra libraries. 
These subtle implementation differences are exemplified in the next section.

\section{Numerical simulations}
\label{sec:numerical_simulations}
\begin{figure*}[t]
    \centering
    \vspace*{-2em}
    \includegraphics[width=.9\textwidth]{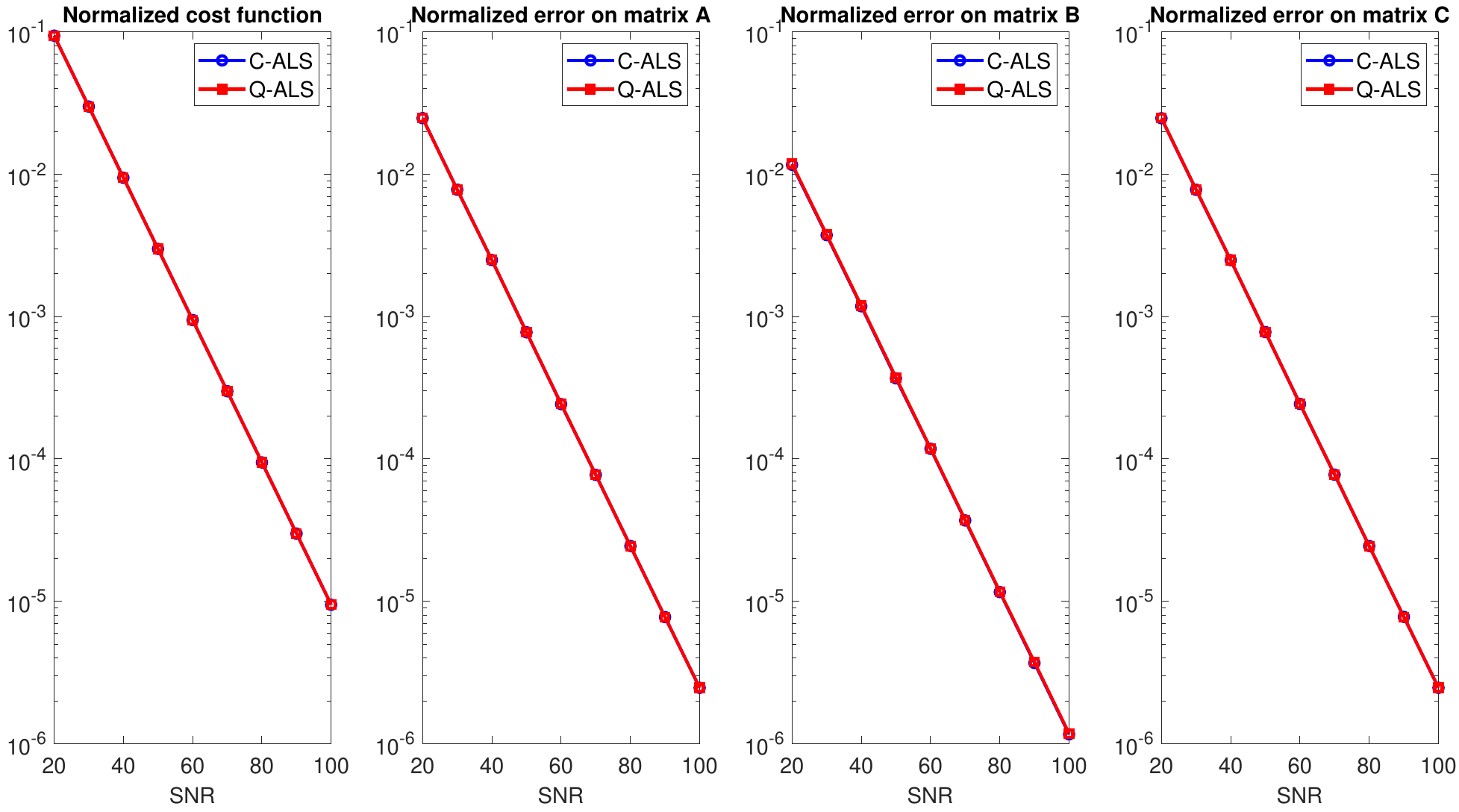}

    \caption{Median value of the cost function and median NMSE on the factor matrices with respect to the SNR. }\label{fig1}
   \end{figure*}

The purpose of the following numerical simulations is twofold. We thus want to : \textit{i)}  verify that  the Q-CPD can be computed effectively ;  \textit{ii)} assess the performances of the proposed algorithms in terms of convergence speed and accuracy in a simple but classical scenario. 
To this purpose, we consider here the Q-CPD decomposition of 500 third order tensors built according to the Q-CPD model  \eqref{eq:CPDformat_elementwise}, with $F=5$ and $N_1=N_2=N_3=10$.
The entries of the three factor matrices were drawn randomly  according to a normal distribution and are different for each tensor. 
Before computing the Q-CPD, an additive white Gaussian noise was applied to each tensor. 
The noise powers were set in order to obtain the desired Signal to Noise Ratio (SNR).
More complex scenarios involving correlated factors, higher ranks, higher order tensors and further comparison between the algorithms  deserve a specific study and are not in the scope of this paper.
All simulations were performed using \textsc{Matlab} environment. The Q-ALS algorithm has been programmed using the QTFM toolbox\footnote{\url{https://nl.mathworks.com/matlabcentral/fileexchange/76158-quaternion-toolbox-for-matlab}}.
Both Q-ALS and C-ALS algorithms were initialized with  the same (random) values and were stopped when the relative deviation between two consecutive values of the cost function fell under $10^{-8}$ with a maximum of 500 iterations.
At the end of each run, after removing the scaling and permutation ambiguities, the relative root mean-squared errors between the actual and estimated factor matrices were computed for both algorithms.
Eventually, the median values of these errors and of the relative cost function were calculated over the 500 tensors realizations.
This scenario was repeated for different SNR. Results are plotted on Figure \ref{fig1}.

Three main conclusions can be drawn.
First, as expected, the curves corresponding to the two algorithms overlap perfectly.
 Second, the errors decrease linearly with the noise power. This behavior is similar to the standard CPD algorithms used for real or complex tensors. Third,  the error values are small enough to consider that  Q-CPD decomposition  can actually be computed, which validates the possibility of using  this new tensor model in practical applications.
In order to illustrate the practical interest of C-ALS, we also computed for each algorithm the average running time, running time per iteration and number of iterations. 
The average running times for C-ALS and Q-ALS were about $0.078s$ and $0.134s$ respectively. As both algorithms compute the same updates at each iteration (and therefore reach the same number of iterations), this difference is entirely explained by the average running time per iteration: about $0.0016s$ and $0.0027s$ for C-ALS and Q-ALS respectively and by the QTFM toolbox functions, which are not as optimized as native \textsc{Matlab} functions.
These observations hold for all the considered SNR.
Eventually the same experiments were run for $F=10$, leading to similar conclusions (plots not shown).

\section{Conclusion}
\label{sec:concluson}
In this paper, we rigorously extended the notion of tensors to quaternion algebra. The key insight is that a quaternion tensor can be understood as a peculiar type of multilinear form, an $\bbH\bbR$-multilinear form, which preserves the fundamental multilinearity properties expected of such mathematical objects.
Building on this foundation, we extended the Tucker and Canonical Polyadic Decompositions (CPD) to quaternion tensors. In particular, we showed the equivalence of the quaternion CPD (Q-CPD) to a specific complex-valued rank-$(2,2,1)$ decomposition and established sufficient conditions for its uniqueness. These conditions rely on a careful handling of left and right linear independence in the quaternion domain, through the notion of left and right ranks, and left and right Kruskal ranks which are defined herein.
Additionally, we derived two ALS-based algorithms for computing the Q-CPD, one operating in the quaternion domain and the other in the complex domain, demonstrating that the new decomposition can be computed efficiently.
Future work will focus on further evaluation of the Q-CPD algorithms and on exploring their potential applications in data analysis and related fields.

\section*{Acknowledgments}
The authors would like to thank the anonymous reviewer for their valuable comments and constructive suggestions, which have greatly improved the quality and clarity of this manuscript.

\appendix
\section{Proof of \Cref{prop1}}
\label{app:technical}
Let us first define $w(\bfu)$ as the number of non zero entries of the complex vector $\bfu$ and recall the Kruskal permutation lemma \cite{KRUSKAL197795}.
\begin{lemma}[Kruskal permutation lemma]\label{kpl}
Let $\bfB$ and $\widetilde{\bfB}$ be two (real or complex) matrices of the same size and assume that $k_\bfB \geq 2$ (no colinear columns). Let $F$ denote the number of columns.
If for any complex vector $\bfx$ such that $w(\widetilde{\bfB}^\H\bfx)\leq F-\rank{\widetilde{\bfB}}+1$ we have  $w(\bfB^\H\bfx)\leq w(\widetilde{\bfB}^\H\bfx)$ then there exists a non singular diagonal matrix $\bfD$ and a permutation matrix $\bfP$ so that $\widetilde{\bfB}=\bfB\bfP\bfD$.
\end{lemma}
We can now start the proof of proposition \ref{prop1}.
\begin{proof}  
Suppose that there exists an alternative solution of the CONFAC decomposition of $\tens{T}_1$ given by ($\widetilde{\bfU}$, $\widetilde{\bfB}$, $\widetilde{\bfW}$). Then, we have $(\bfT_{1}^B)^\T=		(\widetilde{\bfU}\mathbf{\Psi}\krb\widebar{\widetilde{\bfW}})\mathbf{\Phi}^\T\widetilde{\bfB}^\T=(\bfU\mathbf{\Psi}\krb\widebar{\bfW})\mathbf{\Phi}^\T\bfB^\T.$
From the full column rank assumptions, we immediately verify that 
$\rank{(\bfU\mathbf{\Psi}\krb\widebar{\bfW})\mathbf{\Phi}^\T\bfB^\T}=\rank{\bfB}=F= \rank{(\widetilde{\bfU}\mathbf{\Psi}\krb\widebar{\widetilde{\bfW}})\mathbf{\Phi}^\T	\widetilde{\bfB}^\T}$. Therefore $\rank{\widetilde{\bfB}}\geq F$. Since $\rank{\widetilde{\bfB}}\leq F$ by construction, $\rank{\widetilde{\bfB}}= F$. As a consequence, the
 condition in the lemma becomes simply: for any $\bfx$ such that $w(\widetilde{\bfB}^\T\bfx)\leq 1$ we have  $w(\bfB^\T\bfx)\leq w(\widetilde{\bfB}^\T\bfx)$.
 
If $w(\widetilde{\bfB}^\T\bfx)=0$ then $	(\bfU\mathbf{\Psi}\krb\widebar{\bfW})\mathbf{\Phi}^\T\bfB^\T\bfx=\textbf{0}$. $\bfU\mathbf{\Psi}$ and $\bfW$ being full column rank, $\bfU\mathbf{\Psi}\krb\widebar{\bfW}$ is also full column rank, thus we necessarily have $\bfB^\T\bfx=\textbf{0}$ so that $w(\bfB^\T\bfx)=0=w(\widetilde{\bfB}^\T\bfx)$. The condition in the lemma is verified.

Let us now suppose that $w(\widetilde{\bfB}^\T\bfx)=1$, there exists $i$, $1\leq i\leq F$, so that $\widetilde{\bfB}^\T\bfx=\bfe_i$ where $\bfe_i$ is the $i^{th}$ unit vector of size $F$. 
Denoting $\bfz=\bfB^\T\bfx$ and $\bfM^+$ the Moore-Penrose pseudo-inverse of $\bfM$, we have:
\begin{align}
(\widetilde{\bfU}\mathbf{\Psi}\krb\widebar{\widetilde{\bfW}})\mathbf{\Phi}^\T\bfe_i&=&	(\bfU\mathbf{\Psi}\krb\widebar{\bfW})\mathbf{\Phi}^\T\bfz	\\
(\widetilde{\bfU}\mathbf{\Psi}\kron\widebar{\widetilde{\bfW}})(\bfI_{2F}\krb \bfI_{2F})\mathbf{\Phi}^\T\bfe_i&=&	(\bfU\mathbf{\Psi}\kron\widebar{\bfW})(\bfI_{2F}\krb \bfI_{2F})\mathbf{\Phi}^\T\bfz\\
(\bfU\mathbf{\Psi}\kron\widebar{\bfW})^{+}(\widetilde{\bfU}\mathbf{\Psi}\kron\widebar{\widetilde{\bfW}})(\bfI_{2F}\krb \bfI_{2F})\mathbf{\Phi}^\T\bfe_i&=&	(\bfI_{2F}\krb \bfI_{2F})\mathbf{\Phi}^\T\bfz\\
(\mathbf{\Psi}^{+}\bfU^{+}\widetilde{\bfU}\mathbf{\Psi}\kron\widebar{\bfW}^{+}\widetilde{\bfW})(\bfI_{2F}\krb \bfI_{2F})\mathbf{\Phi}^\T\bfe_i&=&	(\bfI_{2F}\krb \bfI_{2F})\mathbf{\Phi}^\T\bfz.
\end{align}
We then define the two non-singular matrices of size $2F \times 2F$:
$\bfR=\mathbf{\Psi}^{+}\bfU^{+}\widetilde{\bfU}\mathbf{\Psi}$
and $\bfS=\widebar{\bfW}^{+}\widetilde{\bfW}$. It yields
\begin{equation}\label{eqinte}
	(\bfR \krb \bfS)\mathbf{\Phi}^\T\bfe_i=(\bfI_{2F}\krb \bfI_{2F})\mathbf{\Phi}^\T\bfz.
\end{equation}  
On the one hand
\begin{eqnarray}\label{eqg}
(\bfR \krb \bfS)\mathbf{\Phi}^\T\bfe_i&=&\bfr_i \kron \bfs_i+\bfr_{f+i} \kron \bfs_{N+i}\\
&=&\vecop{\bfs_i\bfr_i^\T  +\bfs_{N+i}\bfr_{N+i}^\T }
\end{eqnarray} 
where $\bfr_i$ and $\bfs_i$ denote the $i^{th}$ columns of $\bfR$ ans $\bfS$ respectively.

On the other hand
\begin{eqnarray}\label{eqd}
(\bfI_{2F}\krb \bfI_{2F})\mathbf{\Phi}^\T\bfz&=&\sum_{f=1}^F z_f\left(\begin{pmatrix}\bfe_f \\ \bf0\end{pmatrix} \kron \begin{pmatrix}\bfe_f \\ \bf0\end{pmatrix} + \begin{pmatrix}\bf0 \\ \bfe_f\end{pmatrix} \kron \begin{pmatrix}\bf0 \\ \bfe_f\end{pmatrix} \right)\\
&=&\vecop{\begin{pmatrix}\Diag{\bfz} & \bf0 \\  \bf0 & \Diag{\bfz}\end{pmatrix}}.
\end{eqnarray} 
So that equation \eqref{eqinte} rewrites
 \begin{equation}\label{eqfinal}
	\bfs_i\bfr_i^\T  +\bfs_{F+i}\bfr_{F+i}^\T=\begin{pmatrix}\Diag{\bfz} & \bf0 \\  \bf0 & \Diag{\bfz}\end{pmatrix}.
\end{equation} 
Eventually, we just need to consider the two upper blocks. So that defining
\begin{eqnarray}
    \bfr_i^0&=&\begin{pmatrix}R_{1,i}, \cdots, R_{F,i}\end{pmatrix}^\T\\
    \bfr_{F+i}^0&=&\begin{pmatrix}R_{1,F+i}, \cdots, R_{F,F+i}\end{pmatrix}^\T\\
    \bfr_i^1&=&\begin{pmatrix}R_{F+1,i}, \cdots, R_{2F,i}\end{pmatrix}^\T\\
    \bfr_{F+i}^1&=&\begin{pmatrix}R_{F+1,F+i}, \cdots, R_{2F,F+i}\end{pmatrix}^\T\\
       \bfs_i^0&=&\begin{pmatrix}S_{1,i}, \cdots, S_{F,i}\end{pmatrix}^\T\\
       \bfs_{F+i}^0&=&\begin{pmatrix}S_{1,F+i}, \cdots, S_{F,F+i}\end{pmatrix}^\T
\end{eqnarray}
yields
 \begin{eqnarray}
     \bfs_i^0(\bfr_i^1)^\T+\bfs_{F+i}^0(\bfr_{F+i}^1)^\T&=&\bf0 \label{eqs1}\\
      R_{f,i}\bfs_i^0 + R_{f,F+i} \bfs_{F+i}^0&=&z_f\bfe_f, \ \forall f=1,\cdots, F. \label{eqs2}
 \end{eqnarray}  
Thus there exists $\alpha \in \bbC$ such that 
\begin{equation}
    (R_{f,i}+\alpha)\bfs_i^0 =z_f\bfe_f, \ \forall f=1,\cdots, F.
\end{equation}
We immediately deduce that there is at most one $f$ verifying $z_f\neq 0$. Thereby,
$w(\bfB^\T\bfx) \leq 1$. Hence $w(\bfB^\T\bfx) \leq w(\widetilde{\bfB}^\T\bfx)$
and, again, the condition in the permutation lemma is verified.
We can thus apply it and conclude the proof.
\end{proof}
	
\bibliographystyle{siamplain}\bibliography{refs}

\end{document}